\documentclass[twoside,a4paper,10pt]{amsart}
\usepackage{amsfonts,amssymb,amsmath}

\usepackage{dsfont}

\usepackage{amsthm}
\usepackage[all,cmtip]{xy} 
\usepackage{hyperref}

\usepackage{color}

\usepackage{mathtools}

\newcommand{\set}[1]{\left\{#1\right\}}

\newcommand{\abs}[1]{\left| #1 \right|}

\renewcommand{\epsilon}{\varepsilon}

\newcommand{\cupd}{\overset{.}{\cup}}

\newcommand{\rv}{\mathrm{rv}}

\newcommand{\Rr}{\mathbb{R}}
\newcommand{\Nn}{\mathbb{N}}
\newcommand{\Zz}{\mathbb{Z}}

\newcommand{\Qq}{\mathbb{Q}}

\newcommand{\Aa}{\mathbb{A}}
\newcommand{\Bb}{\mathbb{B}}
\newcommand{\Gg}{\mathbb{G}}
\newcommand{\Gm}{{\mathbb{G}_m}}
\newcommand{\un}{\mathds{1}}

\newcommand{\kk}{\mathbf{k}}

\newcommand{\eL}{\mathbb{L}}
\newcommand{\caL}{\mathcal{L}}

\newcommand{\val}{\mathrm{v}}
\newcommand{\valrv}{\mathrm{v}_{\mathrm{rv}}}

\newcommand{\ACVF}{\mathrm{ACVF}}

\newcommand{\VF}{\mathrm{VF}}
\newcommand{\Var}[1]{\mathrm{Var}_{#1}}
\newcommand{\Varmu}[1]{\mathrm{Var}^{\hat \mu}_{#1}}

\newcommand{\RV}{\mathrm{RV}}

\newcommand{\RVn}[1]{\mathrm{RV}[#1]}

\newcommand{\Isp}{\mathrm{I}_{\mathrm{sp}}}
\newcommand{\RES}{\mathrm{RES}}

\newcommand{\KIRES}{!\mathrm{\mathbf{K}}(\mathrm{RES}_K)}

\newcommand{\Gam}{\Gamma}

\newcommand{\Gamn}[1]{\Gamma[#1]}

\newcommand{\Gamnfin}[1]{\Gamma^{\mathrm{fin}}[#1]}

\newcommand{\eu}{\mathrm{eu}}

\newcommand{\bt}{\mathbf{t}}

\newcommand{\parBb}{{\partial\mathbb{B}}}
\newcommand{\an}{\mathrm{an}}

\newcommand{\Ccr}{{\mathbb{C}r}}

\newcommand{\K}[1]{\mathrm{\bf K}(#1)}

\newcommand{\SH}[1]{\mathrm{SH}_\mathfrak{M}(#1)}
\newcommand{\RigSH}[1]{\mathrm{RigSH}_{\mathfrak{M}}(#1)}
\newcommand{\QUSH}[1]{\mathrm{QUSH}_\mathfrak{M}(#1)}

\newcommand{\DA}[1]{\mathrm{DA}(#1)}
\newcommand{\RigDA}[1]{\mathrm{RigDA}(#1)}

\newcommand{\Mot}{\mathrm{M}}
\newcommand{\Motr}{\mathrm{M}_\Rig}

\newcommand{\Spec}[1]{{\mathrm{Spec}(#1)}}
\newcommand{\Spm}[1]{{\mathrm{Spm}(#1)}}
\newcommand{\Spf}[1]{{\mathrm{Spf}(#1)}}

\newcommand{\Sus}{\mathrm{Sus}}
\newcommand{\Gal}{\mathrm{Gal}}
\newcommand{\Th}{\mathrm{Th}}
\newcommand{\Rig}{\mathrm{Rig}}
\newcommand{\gm}{\mathrm{gm}}

\newcommand{\Sch}{\mathrm{Sch}}

\newcommand{\Hom}{\mathrm{Hom}}
\newcommand{\intHom}{\underline{\mathrm{Hom}}}

\newcommand{\VarRig}{\mathrm{VarRig}}

\newcommand{\chir}{\chi_\Rig}
\newcommand{\chirt}{\widetilde{\chi_\Rig}}
\newcommand{\chirGam}{{\chi_\Rig^\Gam}}
\newcommand{\chirRes}{{\chi_\Rig^\RES}}

\newcommand{\chirp}{\chi'_{\Rig}}
\newcommand{\chirpt}{\widetilde{\chi'_\Rig}}
\newcommand{\chirpGam}{{\chi^{\prime\,\Gam}_\Rig}}
\newcommand{\chirpRes}{{\chi^{\prime\,\RES}_\Rig}}
\newcommand{\chirRV}{{\chi_\Rig^\RV}}
\newcommand{\chirpRV}{{\chi^{\prime\,\RV}_\Rig}}

\newcommand{\spc}{\mathrm{sp}}
\newcommand{\tube}[1]{{]#1[}}
\newcommand{\St}{\mathrm{St}}

\newcommand{\calX}{\mathcal{X}}
\newcommand{\calY}{\mathcal{Y}}

\newcommand{\calU}{\mathcal{U}}
\newcommand{\calS}{\mathcal{S}}
\newcommand{\calO}{\mathcal{O}}
\newcommand{\calF}{\mathcal{F}}
\newcommand{\calE}{\mathcal{E}}
\newcommand{\calM}{\mathcal{M}}
\newcommand{\calL}{\mathcal{L}}
\newcommand{\calC}{\mathcal{C}}
\newcommand{\calT}{\mathcal{T}}
\newcommand{\calR}{\mathcal{R}}
\newcommand{\calD}{\mathcal{D}}

\newcommand{\Cone}{\mathrm{Cone}}

\numberwithin{equation}{section}

\theoremstyle{plain}
\newtheorem{nth}{Theorem}[section]

\newtheorem{Theorem}[nth]{Theorem}
\newtheorem{Proposition}[nth]{Proposition}
\newtheorem{Lemma}[nth]{Lemma}
\newtheorem{Corollary}[nth]{Corollary}

\theoremstyle{definition}
\newtheorem{Definition}[nth]{Definition}
\newtheorem{Lemma-Definition}[nth]{Lemma-Definition}
\newtheorem{Proposition-Definition}[nth]{Proposition-Definition}

\theoremstyle{remark}
\newtheorem{Remark}[nth]{Remark}

\title{Virtual rigid motives of semi-algebraic sets}
\author[Arthur Forey]{Arthur Forey}

\address{Sorbone Universit\' es, UPMC Univ Paris 06,  UMR  7586 CNRS,  Institut de Math\' ematiques de Jussieu-Paris Rive Gauche, 75005 Paris, France}
\email{arthur.forey@imj-prg.fr}

\date{\today}

\begin{document}

\begin{abstract}
Let $k$ be a field of characteristic zero containing all roots of unity and $K=k((t))$. We build a ring morphism from the Grothendieck group of semi-algebraic sets over $K$ to the Grothendieck group of motives of rigid analytic varieties over $K$. It extend the morphism sending the class of an algebraic variety over $K$ to its cohomological motive with compact support. We show that it fits inside a commutative diagram involving Hrushovski and Kazhdan's motivic integration and Ayoub's equivalence between motives of rigid analytic varieties over $K$ and quasi-unipotent motives over $k$ ; we also show that it satisfies a form of duality. This allows us to answer a question by Ayoub, Ivorra and Sebag about the analytic Milnor fiber. 
\end{abstract}

\maketitle


\section{Introduction}

Let $k$ be a field of characteristic zero containing all roots of unity and $K=k((t))$ the field of Laurent series. Morel and Voevodsky build in \cite{morel_voevodsky_A1} the category $\mathrm{SH}(k)$ of stable $\Aa^1$-invariant motivic sheaves without transfers over $k$. More generally for $S$ a $k$-scheme they build the category of $S$-motives $\mathrm{SH}(S)$.  Following insight by Voevodsky, see Deligne's notes \cite{deligne_voevodskys_2001}, Ayoub developed in \cite{ayoub_six_1} a six functors formalism for the categories $\mathrm{SH}(-)$, mimicking Grothendieck's six functors formalism for \'etale cohomology. See also in \cite{cisinski_triangulated_2012} an alternative construction by Cisinski and D\' eglise. For $f : X \to Y$ a morphism of schemes, in addition to the direct image $f_* : \mathrm{SH}(X)\to \mathrm{SH}(Y)$ and pull-back $f^* : \mathrm{SH}(Y)\to \mathrm{SH}(X)$, one has the extraordinary direct image $f_! : \mathrm{SH}(X)\to \mathrm{SH}(Y)$ and extraordinary pull-back $f^! : \mathrm{SH}(X)\to \mathrm{SH}(Y)$.  It allows in particular to define for any $S$-scheme $ f :X\to S$ an object $\Mot_{S,c}^\vee(X)=f_!f^*\un_k\in \mathrm{SH}(S)$, the so-called cohomological motive with compact support of $X$. 

Denote by $\K{\Var{k}}$ the Grothendieck group of $k$-varieties. It is the abelian group generated by isomorphism classes of $k$-varieties, with the scissors relations 
\[
[X]=[Y]+[X\backslash Y]
\] for $Y$ a closed subvariety of $X$. Cartesian product induces a ring structure on $\K{\Var{k}}$. 

As $\mathrm{SH}(k)$ is a triangulated category, we can consider its Grothendieck group $\K{\mathrm{SH}(k)}$, which is the abelian group generated by isomorphism classes of its compact (also called constuctible) objects, with relations
$[B]=[A]+[C]$ whenever there is a distinguished triangle
\[
A\to B\to C\overset{+1}{\to}.
\]
Elements of $\K{\mathrm{SH}(k)}$ are called virtual motives and tensor product on $\mathrm{SH}(k)$ induces a ring structure on $\K{\mathrm{SH}(k)}$. The locality principle implies that the assignment $X\in\Var{k}\mapsto [\Mot_{k,c}^\vee(X)]\in \K{\mathrm{SH}(k)}$ satisfies the scissors relations, hence induces a morphism
\[
\chi_k : \K{\Var{k}}\to\K{\mathrm{SH}(k)}
\]
which is a ring morphism. Such a morphism was first considered by  Ivorra and Sebag in \cite{ivorra_nearby_2013}.

Ayoub builds in \cite{ayoub_rigide} the category $\mathrm{RigSH}(K)$ of rigid analytic motives over $K$, in a similar fashion of $\mathrm{SH}(K)$ but instead of $K$-schemes, he starts with rigid analytic $K$-varieties in the sense of Tate. The analytification functor from algebraic $K$-varieties to  rigid $K$-varieties induce a functor
\[
\Rig^* : \mathrm{SH}(K)\to \mathrm{RigSH}(K).
\]
For any rigid $K$-variety $X$, Ayoub defines $\Mot_\Rig(X)$ and $\Mot_\Rig^\vee(X)$ respectively the homological and cohomological rigid motives of $X$. However to our knowledge there is no general notion of cohomological rigid motive with compact support. 

One can also consider $\K{\VF_K}$ the Grothendieck ring of semi-algebraic sets over $K$. If $X=\Spec{A}$ is an affine variety over $K$, a semi-algebraic subset of $X^\an$ is a boolean combination of subsets of the form $\set{x\in X^\an \mid \val(f(x))\leq \val(g(x))}$, for $f,g\in A$ (where $\val$ is the valuation on $K$). The group $\K{\VF_K}$ is then the abelian group of isomorphism classes of semi-algebraic sets (for semi-algebraic bijections) with relations $[X]=[U]+[V]$ if $X$ is the disjoint union of $U$ and $V$. Locally closed semi-algebraic sets inherits a structure of rigid $K$-varieties and $\K{\VF_K}$ is generated as a group by classes of such sets that are smooth rigid $K$-varieties. We could also consider $\K{\VF_K^\an}$, the Grothendieck group of subanalytic sets over $K$ but they turn out to be isomorphic as a byproduct of Hrushovski and Kazhdan's theory of motivic integration \cite{hrushovski_integration_2006}.

In this situation it is rather natural to ask about the existence of a ring morphism
\[
\chir : \K{\VF_K}\to \K{\mathrm{RigSH}(K)}
\]
extending the morphism $\chi_K : \K{\Var{K}}\to\K{\mathrm{SH}(K)}$. 
If $X$ is an algebraic $K$-variety smooth and connected of dimension $d$, then $[\Mot_{K,c}^\vee(X)]=[\Mot_{K}(X)(-d)]$, where $(-d)$ is the Tate twist (iterated $d$ times). We would like to define for $X$ a quasi-compact rigid $K$-variety smooth and connected of dimension $d$, $\chir([X])=[\Motr(X)(-d)]$. Such classes generate $\K{\VF_K}$, hence if it exists we will get uniqueness of a morphism satisfying such conditions. The main objctive of this paper is to show the existence of such a morphism.

The strategy of proof is to use alternative descriptions of $\K{\VF_K}$ and $\mathrm{RigSH}(K)$, the former being established by Hrushovski and Kazhdan, the latter by Ayoub. Let's describe them briefly. 

From a model-theoretic point of view, semi-algebraic sets over $K$ are definable sets in the (first order) theory of algebraically closed valued fields over $K$. If $L$ is a valued field, with ring of integers $\calO_L$ of maximal ideal $\calM_L$, we set $\RV(L)=L^\times/(1+\calM_L)$. Observe that $\RV$ fits in the following exact sequence, where $\mathbf{k}$ is the residue field and $\Gam$ the value group :
\[
1\to \mathbf{k}^\times\to \RV\to \Gam\to 0.
\]
Working in a two sorted language, with one sort $\VF$ for the valued field and one sort $\RV$, Hrushovski and Kazhdan establish in \cite{hrushovski_integration_2006} the following isomorphism of rings :
\[
\oint : \K{\VF_K}\to \K{\RV_K[*]}/\Isp,
\]
where $\K{\RV_K[*]}$ is the Grothendieck ring of definable sets of $\RV$, the $[*]$ meaning that some grading is taken into account and $\Isp$ is an ideal generated by a single explicit relation, see Section \ref{section-recall-HK}. Set $\hat \mu=\lim_{\leftarrow_n}\mu_n$, with $\mu_n$ the group of $n$-th roots of unity in $k$ and $\K{\Varmu{k}}$ the Grothendieck of varieties equipped with some good $\hat \mu$-action, good meaning that the action factors through some good $\mu_n$-action for some $n$.

The ring $\K{\RV_K[*]}$ can be further decomposed into a part generated by $\K{\Varmu{k}}$ and a part generated by definable subsets of the value group. The latter being polytopes, one can apply bounded Euler characteristic to get a ring morphism
\[
\Theta\circ \calE_c : \K{\RV_K[*]}/\Isp\to \K{\Varmu{k}}.
\]

Ayoub on his side defines the category of quasi-unipotent motives $\mathrm{QUSH}(k)$ as the triangulated subcategory of $\mathrm{SH}(\Gm_k)$ with infinite sums generated by homological motives (and their twists) of $\Gm_k$-varieties of the form
\[
X[T,T^{-1},V]/(V^r-Tf)\to \Spec{k[T,T^{-1}]}=\Gm_k
\]
where $X$ is a smooth $k$-variety, $r\in \Nn^*$, and $f\in \Gam(X,\calO^\times_X)$. Let $q : \Spec{K}\to \Gm_k$ be the morphism defined by $T\in k[T,T^{-1}]\mapsto t\in K=k((t))$. Ayoub shows in \cite{ayoub_rigide} that the functor
\[
\mathfrak{F} : \mathrm{QUSH}(k)\overset{q^*}{\to} \mathrm{SH}(K)\overset{\Rig^*}{\to} \mathrm{RigSH}(K)
\]
is an equivalence of categories, denote by $\mathfrak{R}$ a quasi-inverse. 

We will define a morphism 
\[
\chi_{\hat\mu} :\K{\Varmu{k}}\to \K{\mathrm{QUSH}(k)}
\]
compatible with $\chi_k$ in the sense that it commutes with the morphism $\K{\Varmu{k}}\to \K{\Var{k}}$ induced by the forgetful functor and $1^* : \K{\mathrm{SH}(\Gm_k)}\to \K{\mathrm{SH}(k)}$, where $1 : \Spec{k}\to \Gm_k$ is the unit section, see section \ref{section-chihatmu}.

Here is our main theorem.
\begin{Theorem}
\label{main-thm1}
Let $k$ be a field of characteristic zero containing all roots of unity and  let $K=k((t))$. Then there exists a unique ring morphism 
\[\chir : \K{\VF_K}\to \K{\mathrm{RigSH}(K)}
\] such that for any quasi-compact rigid $K$-variety $X$, smooth and connected of dimension $d$, $\chir(X)=[\Motr(X)(-d)]$.

Moreover, all the squares in the following diagram commute :
\[
\xymatrixcolsep{1pc}
\xymatrix{
\K{\Var{K}} \ar[d]_{\chi_K} \ar[r]^-{} &\K{\VF_K} \ar[d]_{\chir} \ar[r]^-{\oint}_-{\simeq} & \K{\RV_K[*]}/\Isp\ar[r]^-{\Theta\circ \calE_c }&\K{\Varmu{k}} \ar[d]^{\chi_{\hat \mu}}\ar[r]^-{} & \K{\Var{k}} \ar[d]^{\chi_k} \\
\K{\mathrm{SH}(K)} \ar[r]_-{\Rig^*} &\K{\mathrm{RigSH}(K)} \ar[rr]^-{\simeq}_-{\mathfrak{R}} && \K{\mathrm{QUSH}(k)} \ar[r]_-{1^*} &\K{\mathrm{SH}(k)}.
}
\]
\end{Theorem}

Ayoub, Ivorra and Sebag ask in \cite[Remark 8.15]{AIS} about the existence of a morphism similar to $\chir$ and speculate that one should be able to recover from it their comparison result about the motivic Milnor fiber. We will show that it is indeed the case, see below.

Observe also that with this diagram in mind, defining $\chir$ is easy since $\mathfrak{R}$ is an isomorphism, it is showing that it satisfies $\chir(X)=[\Motr(X)(-d)]$ that we will have to prove. We will rely for this on an explicit computation of $\oint[X]$ when a semi-stable formal $R=k[[t]]$-model of $X$ is chosen. 

Two choices are made in this construction. The first is when applying bounded Euler characteristic $\calE_c$, where we also could have used Euler characteristic $\calE$, the second is when we apply the morphism $\chi_{\hat \mu}$, where we can also consider the mophism sending the class of a variety to its homological motive with compact support. Varying these choices leads to define three other ring morphisms 
\[\chirp, \chirt, \chirpt : \K{\VF_K}\to \K{\RigSH{K}}\]
satisfying properties analogous to Theorem \ref{main-thm1}. In particular, we will show that $\chirpt$ also extends the morphism $\chi_K$.

We claim that $\chir(X)$ is the virtual incarnation of an hypothetical cohomological rigid motive with compact support of $X$. Hence we expect some duality to appear. Here is what we prove in this direction. 

\begin{Theorem}
\label{main-thm2}
Let $X$ be a quasi-compact smooth rigid variety, $\calX$ an formal $R$-model of $X$, $D$ a locally closed and proper subset of its special fiber $\calX_\sigma$. Consider the tube $]D[$ of $D$ in $\calX$, it is a (possibly non quasi-compact) rigid subvariety of $X$. 
Then
\[
\chir(]D[)=[\Motr^\vee(]D[)].
\]
In particular, if $X$ is a smooth and proper rigid variety,
\[
\chir(X)=[\Motr^\vee(X)].
\]
\end{Theorem}

To prove this theorem, we will once again rely on a choice of a semi-stable formal $R$-model of $X$ and compute explicitly $[\Motr^\vee(]D[)]$ in terms of homological motives of $]D[$ and some subsets of $]D[$. Our approach is inspired by parts of Bittner's works \cite{bittner_universal_2004} and \cite{bittner_motivic_2005} where she defines duality involutions in $\K{\Var{k}}[\eL^{-1}]$ and shows that a toric variety associated to a simplicial fan satisfies an instance of Poincare's duality. 

Theorem \ref{main-thm2} allows us to answer a question asked by Ayoub, Ivorra and Sebag in \cite[Remark 8.15]{AIS} in relation to the motivic Milnor fiber. Fix $X$ a smooth connected $k$-variety and let $f : X\to \Aa^1_k$ be a non constant morphism. Set $X_\sigma$ to be the closed subvariety of $X$ defined by the vanishing of $f$. Jan Denef and Fran\c cois Loeser define in \cite{denef_loser_motivic_igusa, denef_germs_1999, denef_geometry_2001}, see also \cite{loeser_seattle_2009}, the motivic nearby cycle of $f$ as an element $\psi_f\in \K{\Varmu{X_\sigma}}$. If $x : \Spec{k}\to X_\sigma$ is a closed point of $X_\sigma$, fiber product induce a morphism $x^* : \K{\Varmu{X_\sigma}}\to \K{\Varmu{k}}$, and $\psi_{f,x}=x^*\psi_f\in \K{\Varmu{k}}$ is the motivic Milnor fiber of $f$ at $x$. 

Denef and Loeser justify their definition by showing that known additive invariants associated to the classical nearby cycle functor can be recovered from $\psi_f$ and $\psi_{f,x}$, the Euler characteristic for example.

Ivorra and Sebag study a new instance of such a principle in \cite{ivorra_nearby_2013} where they show (with our notations) that $\chi_{X_\sigma}(\psi_f)=[\Psi_f\un]\in \K{\mathrm{SH}(X_\sigma)}$, where $\Psi_f$ is the motivic nearby cycle functor constructed by Ayoub in \cite[Chapitre 3]{ayoub_six_2}. Literally speaking they only prove it in $\K{\mathrm{DA}^{\mathrm{\acute et}}(X_\sigma,\Qq)}$, but it is observed in \cite[Section 8.2]{AIS} that their result generalize to $\K{\mathrm{SH}(X_\sigma)}$. 

It was first observed by Nicaise and Sebag in \cite{nicaise_motivic_2007} that one can relate the motivic Milnor fiber to a rigid analytic variety. Consider the morphism $\Spec{R}\to \Spec{\Aa^1_k}$ induced by $T\in k[T]\mapsto t\in k[[t]]$. Still denote $X\to \Spec{R}$ the base change of $f$ along this morphism, and let $\calX$ be the formal $t$-adic completion of $X$. For $x\in X_{\sigma}$ a closed point, set $\calF_{f,x}^\an$ the tube of $\set{x}$ in $\calX$. It is the analytic Milnor fiber.

Ayoub, Ivorra and Sebag show in \cite{AIS} that 
\[
[1^*\circ\mathfrak{R}\Motr^\vee(\calF_{f,x}^\an)]=\chi_k(\psi_{f,x})\in \K{\mathrm{SH}(k)}.
\]

 In our context, we have $\Theta\circ\calE_c\circ\oint\calF_{f,x}^\an=\psi_{f,x}\in \K{\Varmu{k}}$, we can see it either by a direct computation using resolution of singularities as in \cite{nicaise_ps_tropical} and \cite{nicaise_tropical_2017} or by adapting results by Hrushovski and Loeser \cite{HL_monodromy}. Now Theorem \ref{main-thm2} shows that $\chir(\calF_{f,x}^\an)=[\Motr^\vee(\calF_{f,x}^\an)]$ hence by Theorem \ref{main-thm1}, 
 \[
[\mathfrak{R}\Motr^\vee(\calF_{f,x}^\an)]=\chi_{\hat\mu}(\psi_{f,x})\in \K{\mathrm{QUSH}(k)}.
\]
We then have generalized the result of Ayoub, Ivorra and Sebag to an equivariant setting. 

The paper is organized as follows. See the beginning of each section for the precise content. Section \ref{section-prelim-motivic-integration} is devoted to what we need on Hruskovski and Kazhdan motivic integration. In Section \ref{section-prelim-motives}, we settle what we will use on motives, rigid analytic geometry and rigid motives. In Section \ref{section-realization-map} we build the realization map $\chir$ and prove Theorem \ref{main-thm1}. The last Section \ref{section-duality} is devoted to duality and the proof of Theorem \ref{main-thm2}.

\subsection*{Acknowledgments}
I would like to thank Fran\c cois Loeser for his constant support during the preparation of this project. I also thank Marco Robalo for numerous discussions and Joseph Ayoub, Florian Ivorra and Julien Sebag. This research was partially supported by ANR-15-CE40-0008 (D\' efig\' eo).

\section{Premilinaries on motivic integration}
\label{section-prelim-motivic-integration}

In this section we will introduce Hrushovski and Kazhdan's theory of motivic integration in Section \ref{section-recall-HK} and use it to define two maps from the Grothendieck group of semi-algebraic sets over $K$ to the equivariant Grothendieck group of varieties over $k$ in Section \ref{section-landing-KVarmu}.

\subsection{Recap on Hrushovski and Kazhdan's integration in valued fields}
\label{section-recall-HK}
We outline here the construction of Hrushovski and Kazhdan's motivic integration \cite{hrushovski_integration_2006}, focusing on the universal additive invariant since that is the only part that we will use. See also the papers \cite{yin_special_2010} and \cite{yin_integration_2011} by Yimu Yin who gives an account of the theory in $\ACVF$. 

We will work in the first order theory $\ACVF$ of algebraically closed valued fields of equicharacteristic zero in the two-sorted language $\calL$. The two sorts are $\VF$ and $\RV$. We put the ring language on $\VF$, with symbols $(0,1,+,-,\cdot)$, on $\RV$ we put the group language $(\cdot, ()^{-1})$, a unary predicate $\kk^\times$ for a subgroup, and operations $+ : \kk^2\to \kk$ where $\kk$ is the union of $\kk^\times$ and a symbol 0. We add also a unary function $\rv : \VF^\times=\VF\backslash\set{0} \to \RV$. 

We will also consider the imaginary sort $\Gam$ defined by the exact sequence
\[
1\to \kk^\times \to \RV\to \Gam\to 0,
\]
together with maps $\valrv : \RV \to \Gam$ and $\val : \VF^\times \to \Gam$. 

If $L$ is a valued field, with valuation ring $\calO_L$ and maximal ideal $\calM_L$, define an $\calL$-structure by $\VF(L)=L$, $\RV(L)=L^\times/(1+\calM_L)$, $\kk(L)=\calO_L/\calM_L$, $\Gam(L)=L^\times/\calO_L^\times$. Note that the valuation ring is definable in this language because $\calO_L^\times=\rv^{-1}(\kk^\times(L))$. 

Fix a field $k$ of characteristic zero containing all roots of unity and  set $K=k((t))$. View $K$ as a fixed base structure, for the rest of the paper, we will only consider $\calL(K)$-structures, where $\calL(K)$ is the language obtained by adjoining to $\calL$ constants symbols for elements of $K$. Any valued field extending $K$ can be interpreted as a $\calL(K)$-structure. Denote $\ACVF_K$ the $\calL(K)$-theory of such algebraically closed valued fields. The theory $\ACVF_K$ admits quantifier elimination in the language $\calL(K)$. Quantifier elimination was first proven by Robinson using a two sorted language, with one sort $\VF$ and one sort $\Gamma$ for the value group, see for example \cite{weispfenning_quantifier_1984}. 

We will use the notation $\VF^\bullet$ for $\VF^n$ for some $n$. 
The $\calL(K)$-definable subsets of $\VF^\bullet$ are semi-algebraic sets, that is boolean combinations of sets of the form
\[
\set{x\in \VF^n\mid f(x)=0}
\]
and
\[
\set{x\in \VF^n\mid \val(f(x))\geq \val(g(x))},
\]
where $f$ and $g$ are polynomials with coefficients in $K$.

We will also consider the theory $\ACVF_K^\an$ in the language $\calL_\an(K)$. This language is an enrichment of $\calL$ where we add symbols for restricted analytic functions with coefficients in $K$, see \cite{lipshitz_uniform_2005} and \cite{cluckers_fields_2011} for details. A maximally complete algebraically closed valued field containing $K$ can be enriched as an $\calL_\an(K)$-structure. Denote $\ACVF^\an_K$ their $\calL_\an(K)$-theory. 

We shall refer to $\caL_\an(K)$-definable subsets of $\VF^\bullet$ as subanalytic sets.

Denote by $\K{\VF_K}$ the free group of $\calL(K)$-definable subsets of ${\VF}^\bullet$, with the following relations :
\begin{itemize}
\item $[X]=[Y]$ if there is a semi-algebraic bijection $X\to Y$
\item $[X]=[U]+[V]$ if $X$ is a the disjoint union $X=U\cupd V$.
\end{itemize}
Cartesian product endows $\K{\VF_K}$ with a ring structure.

Denote by $\K{\VF_K^\an}$ the analogous ring for subanalytic sets. 

\begin{Remark}
Note that this framework allows us to consider general semi-algebraic subsets of $K$-varieties (resp. subanalytic subsets of rigid $k$-varieties) as studied for example by Florent Martin in \cite{martin_cohomology_2014}. We say that $S$ is a semi-algebraic subset $X^\an$, for $X$ a $k$-scheme, if $S$ is a finite union $S=\cup S_i$ such that for every $i$, there is an open affine subset $U_i=\Spec{A_i}$ of $U$ such that $S_i\subseteq U_i^\an$ is defined in $U_i^\an$ by boolean combination of subsets of the form $\set{ y\in U_i^\an\mid \val(f(y))\leq r \val(g(y))}$, with $f,g\in A_i$, $r\in \mathbb{Q}$. Hence we can consider its class $[S]\in \K{\VF_K}$. 
\end{Remark}

As $\ACVF_K^\an$ is an enrichment of $\ACVF_K$, we have a canonical map
\[
\K{\VF_K}\to \K{\VF_K^\an}. 
\]
We shall see later that Hrushovski-Kazdhan motivic integration implies that it is an isomorphism. Note that even injectivity is not obvious : it means that for any subanalytic bijection between two semi-algebraic sets, we can find a semi-algebraic bijection between them. 
 
Hrushovski and Kazdhan use a slightly different definition for $\K{\VF_K}$. They define it as the group generated by isomorphism classes of definable sets $X\subseteq \VF^\bullet\times \RV^\bullet$, such that for some $n\in \Nn$, there is some definable function $f : X\to \VF^n$ with finite fibers, with cut-and-paste relations (the function $f$ is not part of the data).  We can show that for such an $X$, there is some definable $X'\subseteq \VF^\bullet$, with a definable bijection $X\simeq X'$. Hence both definitions leads to the same group. Nevertheless, this alternative presentation is very usefull for defining motivic integration, since it amounts to relate $\K{\VF_K}$ and some group related to definable sets in the sort $\RV$, that we will now define.

Define $\RV_K$ to be the category of objects $Y\subseteq \RV$, with definable functions as morphisms. The category $\RV_K[n]$ is the category of pairs $(Y,f)$, with $Y\subseteq \RV^\bullet$ definable and $f : Y\to \RV^n$ a definable finite-to-one function. An morphism between $(Y,f)$ and $(Y',f')$ is a definable function $g : Y\to Y'$. The category $\RES_K$ is the full subcategory of $\RV_K$ whose objects $Y$ satisfy $\valrv(Y)$ finite. One defines similarly $\RES_K[n]$ to be the full subcategory of $\RVn{n}$ whose objects $(Y,f)$ satisfy $\valrv(Y)$ finite. 

Note that the definition of morphisms in $\RV_K[n]$ implies that $\RV_K[n]$ is equivalent to the category of definable sets $X\subseteq \RV^\bullet$ such that there exist a definable function $f : X \to \RV^n$ with finite fibers.

From those categories one form the graded categories 
\[
\RV_K[\leq n]:=\bigoplus_{0\leq i\leq n} \RV_K[i],
\]
\[
\RV_K[*]:=\bigoplus_{i\in \Nn} \RV_K[i],
\]
\[
\RES_K[\leq n]:=\bigoplus_{0\leq i\leq n} \RES_K[i],
\]
\[
\RES_K[*]:=\bigoplus_{i\in \Nn} \RES_K[i].
\]

For later purpose, we will also need a category related to the value group.  One defines $\Gamn{n}$ to be the category with objects subsets of $\Gam^n$ defined by piecewise linear equations and inequations with $\Zz$-coefficients and parameters in $\Qq$. A morphism between $Y$ and $Y'$ is a bijection defined piecewise by composite of $\Qq$-translations and $\mathrm{GL}_n(\Zz)$ morphisms. From this one forms 
\[
\Gamn{*}:= \bigoplus_{n\in \Nn} \Gamn{n}. 
\]
One defines also $\Gamnfin{n}$ and $\Gamnfin{*}$ to be the full subcategories of $\Gamn{n}$ and $\Gamn{*}$ whose objects are finite. 

Each of these categories $\calC$ has disjoint union, induced by disjoint union of definable sets. We can form the associated Grothendieck group $\K{\calC}$. It is the abelian group generated by isomorphism classes of objects of $\calC$, with relations induced by disjoint union. 

Let us remark that for a fixed definable set $X\subseteq \RV^m$, we can view $X$ as an object in $\RV_K[n]$, for any $n\geq m$. Hence for each $n\geq m$, $X$ induce a class 
\[
[X]_n\in \K{\RV_K[n]}\subset \K{\RV_K[*]}=\bigoplus_{k\geq 0} \K{\RV_K[k]}.
\]
If $X$ is non-empty, we then have $[X]_n\neq [X]_{n'}$ for $n\neq n'$. 

The cartesian product induces ring structure on $\K{\RV_K}$ and $\K{\RES_K}$
If $(X,f)\in \RVn{n}$ and $(X',f')\in \RVn{n'}$, then $(X\times X',f\times f')\in \RVn{n+n'}$, inducing a structure of graded ring on $\K{\RV_K[*]}$ and $\K{\RES_K[*]}$.

Let $(X,f)\in \RV_K[n]$. Define $\mathfrak{L}(X,f)$ as a fiber product
\[
\mathfrak{L}(X,f)=\set{(x,y)\in \VF^n\times X\mid \rv(x)=f(y)}.
\]
As $f$ is finite-to-one, the projection of $\mathfrak{L}(X,f)$ to $\VF^n$ is finite-to-one, hence we can view it as an object in $\VF_K$. 

It turns out that if $(X,f), (X',f')\in \RV_K[n]$, with a definable bijection $X\simeq X'$, then there is a definable bijection $\mathfrak{L}(X,f)\simeq \mathfrak{L}(X',f')$, hence we have a ring morphism 
\[
\K{\RV_K[*]}\to \K{\VF_K}.
\]
For example, $\mathfrak{L}[1]_n$ is the class of an open polydisc of dimension $n$. 

Hrushovski and Kazhdan show that $\mathfrak{L}$ is surjective and study the kernel. 
Let $\RV^{>0}=\set{x\in \RV\mid v_\rv(x)>0}$. Then $\mathfrak{L}[\RV^{>0}]_1= [\set{x\in \VF^\times\mid \val(x)>0}]$, the open unit ball without zero. Hence in $\K{\VF_K}$, we have 
\[ \mathfrak{L}[\RV^{>0}]_1 + \mathfrak{L}[1]_0 =\mathfrak{L}[1]_1.\]
This relation generate the whole kernel of $\mathfrak{L}$. 
Denote by $\Isp$ the ideal of $\K{\RV_K[*]}$ generated by $[\RV^{>0}]+[1]_0-[1]_1$. 

The main theorem of \cite{hrushovski_integration_2006} is the following. 
\begin{Theorem}
\label{int-HK-additive}
The morphism $\mathfrak{L}$ is surjective and its kernel is $\Isp$. 
\end{Theorem}
Denote by $\oint$ its inverse :
\[
\oint : \K{\VF_K}\to \K{\RV_K[*]}/\Isp.
\]

This morphism is what we call (additive) motivic integration. Hrushovski and Kazhdan prove this isomorphism at the level of semi-rings, but we will not need it here. They also show that we can add volume forms to the various categories and still get an isomorphism. They also show that this isomorphism is defined for any first order theory $T$ which is $V$-minimal. The two main examples of such theories being $\ACVF_K$ and $\ACVF_K^\an$, we get also an isomorphism
\[
\oint^\an : \K{\VF_K^\an}\to \K{\RV_K^\an[*]}/\Isp.
\]
But quantifier elimination shows that $\K{\RV_K^\an[*]}\simeq \K{\RV_K[*]}$, hence in particular $\K{\VF_K}\simeq \K{\VF_K^\an}$.

\subsection{Landing in $\K{\Varmu{k}}$}
\label{section-landing-KVarmu}

Our goal here is to relate the target ring of motivic integration $\K{\RV[*]}/\Isp$ to something more algebraic, namely the Grothendieck group of $k$-varieties equipped with a $\hat \mu$-action, where $\hat \mu=\lim_{\underset{n}{\leftarrow}}\mu_n$ and $\mu_n$ the group of $n$-th roots of unity.

Let first show that $\K{\RV_K[*]}$ splits into a $\RES$ part and a $\Gamma$ part. 

The map $Y\in \Gamn{n}\mapsto \valrv^{-1}(Y)\in \RV_K[n]$ induces a functor $\Gamn{n}\to \RV_K[n]$, because the morphisms in $\Gamn{n}$ are piecewise $\mathrm{GL}_n(\Zz)$ transformations, hence lift to $\RV_K[n]$ morphisms. One also gets a functor $\Gamnfin{n}\to \RV_K[n]$, whose image lies in $\RES_K[n]$. Hence one gets also functor $\Gamnfin{*}\to \RES_K[*]$. At the level of Grothendieck groups, one then can see $\K{\Gamnfin{n}}$ as a subgroup of $\K{\Gamn{n}}$ and $\K{\RES_K[n]}$. Hence the map $\K{\RES_K[*]} \otimes \K{\Gamn{*}}\to \K{\RV_K[*]}$ defined by $[X]\otimes[Y] \mapsto [X\times \valrv^{-1}(Y)]$ induces a map 
\[\K{\RES_K[*]} \otimes_{\K{\Gamnfin{*}}} \K{\Gamn{*}}\to \K{\RV_K[*]}.\]

\begin{Proposition}[{\cite[Corollary 10.3]{hrushovski_integration_2006}}] 
\label{prop-iso-RV-RES-Gam}
The map $\K{\RES_K[*]} \otimes_{\K{\Gamnfin{*}}} \K{\Gamn{*}}\to \K{\RV_K[*]}$ is an isomorphism of rings.
\end{Proposition}

As the theory of $\Gamma$ is o-minimal, one can use o-minimal Euler characteristic to get an additive map $\eu : \K{\Gamn{n}}\to \Zz$. Any $X\subseteq \Gam^n$ can be finitely partitioned into pieces definably isomorphic to open cubes $(0,1)^k$, one set $\eu((0,1)^k)=(-1)^k$ and then defines $\eu(X)$ by additivity. One can show that this does not depends on the chosen partition of $X$, see \cite[Chapter 4]{van_den_dries_tame_1998}. One can also show that when $M\to +\infty$, $\eu(X\cap [-M,M]^n)$ stabilizes and one defines the bounded Euler characteristic to be
\[
\eu_c(X):=\lim_{M\to +\infty}\eu(X\cap [-M,M]^n).
\]

Using those Euler characteristics, one can now get rid of the $\Gamma$-part in $\K{\RV_K[*]}$. 

Recall that $\mathrm{I}_{\mathrm{sp}}$ is the ideal of $\K{\RV_K[*]}$ spanned by the element $[1]_1-[1]_0-[\RV^{>0}]_1$. For $a\in \Qq$, denote also $e_a=[\valrv^{-1}(a)]_1\in \K{\RES_K[1]}$. Let $!\mathrm{I}$ the ideal of $\K{\RES_K}$ spanned by all differences $e_a-e_0$ and denote $\KIRES:=\K{\RES_K}/!\mathrm{I}$. Define also $\eL=[\Aa^1_k]$.

\begin{Proposition}[{\cite[Theorem 10.5 (2)]{hrushovski_integration_2006}}] 
\label{prop-def-calE}
There is a ring morphism \[\mathcal{E}:\K{\RV_K[*]}/\Isp\to \KIRES[\eL^{-1}],\] with $\mathcal{E}([X]_k)=[X]/\eL^k$ for $X\in \RES_K[n]$, $\mathcal{E}([\RV^{>0}]_1)=[{\mathbb{G}_m}]/\eL$. 
\end{Proposition}
\begin{proof}
From the ring isomorphism $\K{\RES_K[*]} \otimes_{\K{\Gamnfin{*}}} \K{\Gamn{*}}\to \K{\RV_K[*]}$, it suffices to define the image of $a\otimes b$, for $a\in \K{\RES_K[r]}$ and $b\in \K{\Gamn{s}}$. We set $\mathcal{E}(a\otimes b)=\eu(b)a \cdot [{\Gm}]^s/\eL^{r+s}\in \KIRES[\eL^{-1}]$. As the product is defined by cartesian product, we check that this indeed defines a ring morphism. If $a\in {\K{\Gamnfin{*}}}$, the relations involved in $!I$ show that $\mathcal{E}(a\otimes 1)=\mathcal{E}(1\otimes a)$, hence this defines a map $\K{\RES_K[*]} \otimes_{\K{\Gamnfin{*}}} \K{\Gamn{*}} \to \KIRES[\eL^{-1}]$. Moreover, as 
\[\mathcal{E}([\RV^{>0}]_1)+\mathcal{E}([1]_0)=-\frac{[\Gg_m]}{\eL}+1=\eL^{-1}=\mathcal{E}([1]_1),\]
the generator of $\Isp$ is sent to 0 hence it induces a map 
\[\mathcal{E}:\K{\RV_K[*]}/\Isp\to \KIRES[\eL^{-1}].\]
\end{proof}

\begin{Proposition}[{\cite[Theorem 10.5 (4)]{hrushovski_integration_2006}}] 
There is a ring morphism 
\[\mathcal{E}_c:\K{\RV_K[*]}/\Isp\to \KIRES,\]
 with $\mathcal{E}_c([X]_k)=[X]$ for $X\in \RES_K[n]$ and $\mathcal{E}_c(\val^{-1}(\Delta))=\eu_c(\Delta)$.
\end{Proposition}
\begin{proof}
From the ring isomorphism $\K{\RES_K[*]} \otimes_{\K{\Gamnfin{*}}} \K{\Gamn{*}}\to \K{\RV_K[*]}$, it suffices to define the image of $a\otimes b$, for $a\in \K{\RES_K[r]}$ and $b\in \K{\Gamn{s}}$. We set $\mathcal{E}_c(a\otimes b)=\eu(b)a \cdot [{\mathbb{G}_s}]^s\in \KIRES$. As the product is defined by cartesian product, we check that this indeed defines a ring morphism. If $a\in {\K{\Gamnfin{*}}}$, the relations involved in $!I$ show that $\mathcal{E}_c(a\otimes 1)=\mathcal{E}_c(1\otimes a)$, hence this defines a map $\K{\RES_K[*]} \otimes_{\K{\Gamnfin{*}}} \K{\Gamn{*}} \to \KIRES$. Moreover, as $\mathcal{E}_c([\RV^{>0}]_1)=0$, 
\[
\mathcal{E}_c([\RV^{>0}]_1+[1]_0)=1=\mathcal{E}_c([1]_1),
\]
the generator of $\Isp$ is sent to 0 hence it induces a map 
\[\mathcal{E}_c:\K{\RV_K[*]}/\Isp\to \KIRES.\]
\end{proof}
\begin{Remark}
If one mods out by $[\Aa^1]-1$, then the two morphisms $\mathcal{E}$ and $\mathcal{E}_c$ coincide. However, in our situation, one cannot change one for the other. 
\end{Remark}

Let $\mu_n$ the group of $n$-th roots of unity and $\displaystyle{\hat \mu=\lim_{\underline{n}{\leftarrow}} \mu_n}$.
Define $\Varmu{k}$ to be the category of quasi-projective $k$-varieties equipped with a good $\hat \mu$-action, that is, a $\hat \mu$-action that factors through some good $\mu_n$-action. By a good $\mu_n$ action, we mean an action such that the orbit of every point is contained in an affine open subset stable by the action. Let $\K{{\Varmu{k}}^{\flat}}$ the abelian group generated by isomorphism classes of quasi-projective $k$-varieties $X$ equipped with good $\hat\mu$-action, with the scissors relations. Let $\K{{\Varmu{k}}}$ the quotient of $\K{{\Varmu{k}}^{\flat}}$ by additional relations $[(E,\rho)]=[(E,\rho')]$ if $E$ is a finite dimensional $k$-vector space and $\rho$, $\rho'$ two good linear $\hat\mu$-actions on $V$. Note that cartesian product induces ring structures on $\K{{\Varmu{k}}^{\flat}}$ and $\K{{\Varmu{k}}^{\flat}}$. 

We want to define a map $\K{\RES_K}\to \K{{\Varmu{k}}^{\flat}}$. Fix a set of parameters $t_a\in K((t))^{\mathrm{alg}}$ for $a\in \Qq$ such that $t_1=t$ and $t_{ab}=t_b^a$ for $a\in \Nn^*$ and denote $\bt_a:=\rv(t_a)$. 
Set $V^*_\gamma=\val^{-1}(\gamma)$ and $V_\gamma=V_\gamma^*\cup\set{0}$. If $X\in \RES$, then $X\subseteq \RV^n$ and the image of $\valrv : X\to \Gamma^n$ is finite. Working piecewize we can suppose this image is a singleton.  In this case, there are $m, k_1,...,k_n\in \Nn^*$ such that $X\subseteq V_{k_1/m}\times...\times V_{k_n/m}$. The function $g : (x_1,...,x_n)\in X \mapsto (x_1/\bt_{k_1/m},...,x_n/\bt_{k_n/m})\in k^n$ is $K((t^{1/m}))$-definable and its image $g(X)$ inherits a $\mu_n$-action from the one on $X$. Moreover $g(X)$ is a definable subset of $k^n$, hence constructible by quantifier elimination. So we get a map $\Theta : \K{\RES_K}\to \K{{\Varmu{k}}^{\flat}}$, and it induces also a map $ \KIRES\to \K{\Varmu{k}}$.
Hrushovski and Loeser prove in \cite{HL_monodromy} the following proposition. 
 \begin{Proposition}
 \label{iso-res-varmu}
 The ring morphisms 
 \[\Theta : \K{\RES_K}\to \K{{\Varmu{k}}^{\flat}}\]
 and 
 \[
\Theta : {!\K{\RES_K}}\to \K{\Varmu{k}}
 \]
 are isomorphisms.
 \end{Proposition}
 \begin{proof}
As a linear $\mu_m$-action on $k^n$ is diagonalisable, the relations added when dropping the $\flat$ corresponds to $!I$, hence it suffices to show the first isomorphism. Let first prove that it is injective. By scissors relations, it suffices to show that if $\Theta(X)=\Theta(Y)$ for some $X\subseteq \RV^n$ and $Y\subseteq \RV^{n'}$ such that $\valrv(X)$ and $\valrv(Y)$ are singletons, then $[X]=[Y]$. Consider $g, g'$ as in the definition of $\Theta$, suppose $g(X)$ and $g(Y)$ are equipped with $\mu_m$ actions and pick $f$ a $\mu_m$-invariant definable bijection between $g(X)$ and $g'(Y)$. Then $g'^{-1}\circ f\circ g$ is a $k((t^{1/m}))$-definable bijection between $X$ and $Y$ that is invariant under $\mu_n=\Gal(k((t^{1/m}))/k((t)))$-action, hence it is $k((t))$-definable, that is $[X]=[Y]$. 

Now we prove surjectivity. Fix $X$ a quasi-projective $k$-variety with a $\mu_n$-action. By induction on the dimension of $X$, it suffices to find $W\in \RES_K$ with $\Theta(W)$ dense in $X$. Up to partitioning, one can moreover assume the kernel of the action is trivial, that is the $\mu_n$ action is faithful (for possibly a smaller $m$). Consider the quotient $U=X/\mu_m$. Then the Galois group of $k(X)$ over $k(U)$ is $\mu_m$ and by Kummer theory there is an $f\in k(U)$ such that $k(X)=k(U)(f^{1/m})$. Up to neglecting a part of smaller dimension, one can moreover assume that $f$ is regular and non-vanishing, that is $f\in \Gamma(U,\mathcal{O}_U^\times)$. If we set $W=\set{(u,v)\in U\times V_{1/m}\mid v^m=\bt f(u)}$, then we have  $\Theta(W)=X$.
 \end{proof}
 
 If $U\subseteq \Aa^n_k$ is a smooth subvariety of $\Aa^n_k$, $f\in \Gamma(U,\mathcal{O}_U^\times)$ an invertible regular function on $U$ and $r\in\Nn\backslash\set{0}$, denote
 \[
 Q^\RV_r(U,f)=\set{(u,v)\in V_0^n\times V_{1/r}\mid u\in U, v^r=\bt f(u)}.
 \]
 
\begin{Corollary}
\label{cor-QRES-generators-KRES}
The ring $\K{\RES_K}$ is generated by classes of  $Q^\RV_r(U,f)$. The ring $\K{\RES_K[*]}$ is generated by classes of $[Q^\RV_r(U,f)]_n\in \K{\RES_K[n]}$.
\end{Corollary}

\begin{Corollary}
\label{cor-varmu-vargm}
There is an map
\[
\K{{\Varmu{k}}} \rightarrow   \K{ \Var{\Gm_k} }.
\]
\end{Corollary}
\begin{proof}
From the proof of Proposition \ref{iso-res-varmu}, $\K{{\Varmu{k}}^\flat}$ is generated by classes of the form 
\[
Y=X[V]/(V^m-f),
\]
for $X$ a $k$-variety, $f\in \Gamma(X,\mathcal{O}_X^\times)$ with the $\mu_m$-action induced by multiplication of $V$ by $\zeta\in \mu_m$. To such an class, one associate the class of
\[
Z=X[V,V^{-1},T,T^{-1}]/(V^m-Tf)\to \Gm_k=\Spec{k[T,T^{-1}]}.
\]
in $\K{ \Var{\Gm_k} }$.
By induction on the dimension as in the proof of proposition, this leads to a well-defined map 
\[\K{{\Varmu{k}}^\flat}\to \K{ \Var{\Gm_k} }.\]
Indeed, given $Y$ as above, $X$ and $m$ are uniquely determined. The function $f$ is only determined up to a factor in ${\mathcal{O}_X^\times}^n$, but all different choices of representatives will lead to isomorphic $Z$. 

Finally, note that the relations added when dropping the flat are in the kernel of the above map. Indeed, once again, because a linear action of $\mu_n$ on $k^r$ is diagonalisable, it suffices to show that the image of $[(k,\mu_n)]$, where $\mu_n$ acts on $k$ by multiplication by $n$-th roots of unity, is independent of $n$. As 0 is a fixed point, we can restrict the action on $k^\times$. The image in $\K{ \Var{\Gm_k} }$ is then $[\Spec{k[U,U^{-1},T,T^{-1},V]/(V^n-TU)}]$. But this variety is isomorphic to $\Spec{k[V,V^{-1},T,T^{-1}]}$ over $\Gm_k$, by the map defined by $U\mapsto V^nT^{-1}$, $V\mapsto V$. 
\end{proof}

\section{Preliminaries on motives}
\label{section-prelim-motives}

This section is devoted to fix notations about motives. After a brief recap on triangulated categories in Section \ref{section-triang-cat}, we will outline the six functors formalism in Section \ref{section-six-functors}.  We then built a map from the equivariant Grothendieck group of varieties to the Grothendieck group of quasi-unipotent motives in Section \ref{section-chihatmu}. Finally, we give some background on rigid analytic geometry and formal schemes in Sections \ref{section-rigid-geometry} and \ref{section-formal-schemes} before focusing on motives of rigid analytic varieties in Section \ref{section-rigid-motives}.

\subsection{Triangulated categories}
\label{section-triang-cat}

A triangulated category, as introduced by Verdier in his thesis  \cite{verdier_categories_1996}, is an additive category endowed with an equivalence, denoted $-[1]$ and called the suspension, and a class of distinguished triangles, of the form 
\[
A\to B\to C\overset{+1}{\to},
\]
satisfying some axioms. 

To every arrow $f : A\to B$, we can find a $g : B\to C$ such that
\[
A\to B\to C \overset{+1}{\to}
\] 
is a distinguished triangle. Such a $C$ is unique up to non-cannonical isomorphism. We call $C$ the (mapping) cone of $f$,  written $\Cone(f)$. 
We have that $f$ is an isomorphism if and only $\Cone(f)=0$. 

If $\calT$ is a triangulated category, and $\calS$ a full subcategory closed under suspension and desuspension, Verdier shows that one can define a quotient triangulated category $\calT/\calS$ which is universal with respect to the following properties :
\begin{itemize}
\item There is a canonical triangulated functor $\calT\to \calT/\calS$  which is the identity on objects. 
\item For every $A\in \calS$, $A\simeq 0$ in $\calT/\calS$. 
\end{itemize} 

The idea is to consider the class of arrows $\mathrm{Ar}(\calS)=\set{\alpha : A\to B \mid \Cone(\alpha)\in \calS}$ and formally invert them inside $\calT$. We call $\calT/\calS$ the Verdier localization of $\calT$ with respect to $\calS$.

Assume $\calT$ is a triangulated category admitting direct sums indexed by arbitrary sets. An object $A$ of $\calT$ is said to be compact if the canonical morphism
\[
\bigoplus_{s\in S}\hom_\calT(A,B_s)\to \hom_{\calT}(A,\bigoplus_{s\in S}B_s)
\]
is an isomorphism for any set $\set{B_s}_{s\in S}$ of objects of $\calT$. Let $\calT_{\mathrm{cp}}$ the full subcategory of compact objects of $\calT$. It is a triangulated subcategory of $\calT$. 

If $\calT$ is a triangulated category, we define its Grothendieck group $\K{\calT}$ as the free abelian group generated by isomorphism classes of objects of $\calT_{\mathrm{cp}}$ with relations
$[B]=[A]+[C]$ for every distinguished triangle 
\[
A\to B\to C\overset{+1}{\to}.
\]
Note that the class $[\Cone(\alpha)]\in \K{\calT}$ of the cone of an arrow $\alpha : A \to B$ is now canonically defined. 

\begin{Remark}
We must restrict to compact objects when we define $\K{\calT}$ to prevent the group to be trivial. Indeed, otherwise if $A$ is an object of $\calT$, we would have $[\bigoplus_{n\in \Nn} A]=[A]+[\bigoplus_{n\in \Nn} A]$ hence $[A]=0$. 
\end{Remark}

As for every compact object $A$, the triangle
\[
A\to 0\to A[1] \overset{+1}{\to}
\]
is distinguished, $[A[1]]=-[A]\in \K{\calT}$, hence the suspension is indempotent in $\K{\calT}$. 

Moreover, as we have an distinguished triangle for every $A,B\in \calT_{\mathrm{cp}}$
\[
A\to A\oplus B\to B \overset{+1}{\to},
\]
we have $[A\oplus B]=[A]+[B]$. 

If $\calT$ is moreover a monoidal triangulated category, then $\K{\calT}$ inherits of a ring structure induced by tensor product.

\subsection{Stable category of motives}
\label{section-six-functors}

Denote by $\SH{S}$ the stable category of motivic sheaves over $S$ for the Nisnevich topology and coefficients $\mathfrak{M}$, as studied by Ayoub in \cite[D\'efinition 4.5.21]{ayoub_six_2}. The two main examples are if $\mathfrak{M}$ is the category of simplicial spectra, in which case $\SH{S}$ is the stable homotopy category (without transfers) of Morel-Voevodsky introduced in \cite{morel_voevodsky_A1}. The other one is if $\mathfrak{M}$ is the category of complexes of $\Lambda$ modules, for some ring $\Lambda$. In this case we denote $\SH{S}=\DA{S,\Lambda}$.

All schemes are separated and of finite type. 
If $X$ is a smooth $S$-scheme, the homological motive of $X$ is $\Mot_S(X)=\Sus_T^0(\Lambda(X))\in \SH{S}$. 

For any $A\in \SH{S}$, the tensor product $-\otimes A$ admits a right adjoint $\intHom(A,-)$, the internal Hom. Denote by $\Mot_S^\vee(X)=\intHom(\Mot_S(X),\un_S)$. It is the cohomological motive of $X$. 

For $r\in \Nn$, denote by $\un_S(r):=\Sus_T(T^{\otimes r})[-2r]$ and $-(-r)$ the functor $-\otimes \un_S(r)$. It is the Tate twist. As $\Sus_T^r(-)\otimes \Sus_T^s(-)=\Sus_T^{r+s}(-\otimes -)$, we see that $(-r)\circ (-s)=(-r-s)$. Set also $-(r)=-\otimes \Sus^r(\un_S)[2r]$. Thanks to the stabilization, $(r)$ is the inverse of $(-r)$, because $\Sus_T^0(T^{\otimes r})\otimes \Sus_T^r(\un_S)=\Sus_T^r(T^{\otimes r})=\Sus_T^0(\un_S)$.

The categories $\SH{-}$ possess various fonctorialities. If $f : X\to Y$ is a morphism of schemes, then the pull-back $f^*$ and the push-forward $f_*$ defined at the level of sheaves induce functors $f^* : \SH{Y}\to \SH{X}$ and $f_* : \SH{X}\to \SH{Y}$, $f_*$ is a right adjoint to $f^*$. If $f$ is smooth, then $f^*$ also possesses a left adjoint denoted $f_\sharp$, moreover, for any commutative diagram
\[
\xymatrixcolsep{5pc}
\xymatrix{
\bullet \ar[d]_{f'} \ar[r]^-{g'} &\bullet \ar[d]^{f} \\
\bullet \ar[r]_-{g} &\bullet,
}
\]
there is a 2-isomorphism (\emph{i.e.} an invertible natural transformation) between $f'_\sharp g'^*$ and $g^*f_\sharp$. 

\emph{Locality.}
Let $X$ some $S$-scheme, $i : Z \to X$ a closed immersion and $ j : U \to X$ the complementary open immersion.  Then the pair $(i^*, j^*)$ is conservative, that is for any $A\in \SH{X}$, $A=0$ if and only if $i^*A=0$ and $j^*A=0$. Moreover, the counit of the adjuction $i^*i_*\to \mathrm{id}$ is an isomorphism. 

 From this property one can deduce that for any $A\in \SH{X}$, there is a distinguished triangle
\begin{equation}
\label{eqn-localitytriangle}
j_\sharp j^* A \to A  \to i_*i^*A\overset{+1}{\to}.
\end{equation}
In this situation, we set $i_!=i_*$ and $j_!=j_\sharp$, such that the triangle \ref{eqn-localitytriangle} becomes
\begin{equation}
\label{eqn-localitytrianglebis}
j_! j^* A \to A  \to i_!i^*A\overset{+1}{\to}.
\end{equation}

We also set $j^!=j^*$ and $i^!(A)=i^*\mathrm{Cone}(A\to j_*j^*A)[-1]$ (one needs to check that this construction is functorial), such that we have the dual triangle 

\begin{equation}
\label{eqn-localitytrianglebisdual}
i_* i^! A \to A  \to j_*j^!A\overset{+1}{\to}.
\end{equation}

\emph{Thom equivalences.}
Let $\calE$ be a locally free sheaf over $X$. Identifying $\calE$ with the total space of the vector bundle associated to it, set 
$p : \calE \to X$ the projection and $s : X\to \calE$ the zero section. Define $\Th(\calE)=p_\sharp\circ s_*$ the Thom equivalence associated to $\calE$. Thanks to the stabilization, it is an autoequivalence of $\SH{X}$. The exact triangle given by locality also show that $\Th(\calE)(\un_X)=\calE/(\calE\backslash s(X))$, the maybe more familiar definition of Thom space. 

If $\calE\simeq \calO^n_S$ is free, then $\Th(\calE)[-2n]$ is the Tate twist $(n)$.

 For $f$ smooth, denote $\Omega_f$ the relative canonical bundle of $f$ and set $f_!=f_\sharp\circ \Th^{-1}(\Omega_f)$ and $f^!=\Th(\Omega_f)\circ f^*$.

 For $f$ quasi-projective, we can factor (non-uniquely) $f$ as $f=g\circ i$ where $g$ is smooth and $i$ is a closed immersion. In this situation, we set $f_!=g_!i_*$ and $f^!=i^!g^!$.  One of the main results of \cite{ayoub_six_1} is that such a construction is independent of the choice of the factorization $f=g \circ i$. 
 
For any $f$, $f_!$ is a left adjoint to $f^!$. For $f$ projective, we have $f_!=f_*$.
 
One can drop the assuption of quasi-projectivity. Using Nagata compactification theorem, if $f$ is a morphism separated of finite type, we can factor $f$ as $f=p\circ j$ with $j$ an open immersion and $p$ proper and define $f_!=p_*\circ j_\sharp$. Once again, one can show that this does not depend on the choice of the compactification, see \cite{cisinski_triangulated_2012} for details.

 \begin{Definition}
 We can now define for any $f : X\to S$ the homological motive of $X$ as 
 \[\Mot_S(X)=f_!f^!(\un_S),\]
 the cohomological motive of $X$ as 
 \[\Mot^\vee_S(X)=f_*f^*(\un_S),\]
 the cohomological motive with compact support of $X$ as
  \[\Mot^\vee_{S,c}(X)=f_!f^*(\un_S)\]
  and the homological motive with compact support of $X$ as
    \[\Mot_{S,c}(X)=f_*f^!(\un_S).\]
\end{Definition}
Observe that such a definition is compatible with the one given above for $f$ smooth. Indeed, in that case we have $f_!=f_\sharp \Th^{-1}(\Omega_f)$ and $f^!=\Th(\Omega_f)f^*$ hence $f_!f^!=f_\sharp f^*$.

If $f$ is proper, observe $\Mot_{S,c}^\vee(X)=\Mot_{S}^\vee(X)$ because $f_!=f_*$.

\begin{Proposition}[{\cite[Scholie 1.4.2]{ayoub_six_1}}]
Suppose there is a cartesian square
\[
\xymatrixcolsep{5pc}
\xymatrix{
\bullet \ar[d]_{f'} \ar[r]^-{g'} &\bullet \ar[d]^{f} \\
\bullet \ar[r]_-{g} &\bullet
}
\]
Then there is a 2-isomorphism between $f'_!g'^*$ and $g^*f_!$. 
\end{Proposition}

The following proposition is the projection formula. 
\begin{Proposition}[{\cite[Proposition 2.3.40]{ayoub_six_1}}]
Let $f : Y\to X, A\in \SH{X}, B\in\SH{Y}$. Then
\[
f_!(f^*A\otimes_Y B)=A\otimes_X f_! B. 
\]
\end{Proposition}

The subcategory $\SH{S}_{\mathrm{cp}}$ of compact objects of $\SH{S}$ is the smallest triangulated subcategory  of $\SH{S}$ stable by direct factors and generated by homological motives of smooth quasi-projective varieties and their twists. 

The following motivic realisation has already been considered by Ivorra and Sebag in \cite[Lemma 2.1]{ivorra_nearby_2013}.
\begin{Proposition}
\label{realisation-motives-is}
Let $S$ be a $k$-scheme. There is a unique ring morphism \[\chi_S: \K{\Var{S}} \to \K{\SH{S}}\] such that $\chi_S([X])=\Mot_{S,c}^\vee(X)$ for any $S$-scheme $f : X \to S$. 
\end{Proposition}
\begin{proof}
To show that $\chi_S$ is a group morphism well defined on $\K{\Var{S}}$, we need to check that for any $S$-scheme $X$,  if $i : Y\hookrightarrow X$ is a closed immersion, with $j :U\hookrightarrow X$ the open complement immersion, then $[\Mot_{S,c}(X)]=[\Mot_{S,c}(Y)]+[\Mot_{S,c}(U)]$ in $\K{\SH{S}}$. 
For any object $A\in \K{\SH{X}}$ we have the distinguished triangle
\[
j_!j^*A\to A\to i_!i^*A\overset{+1}{\to}
\]
By applying $f_!$ to this triangle and substituting $A$ by $f^*{\un}_S$, we get the required distinguished triangle
\[
\Mot_{S,c}^\vee(U)\to \Mot_{S,c}^\vee(X)\to \Mot_{S,c}^\vee(Y)\overset{+1}{\to}.\]

It remains to show that it is compatible with ring structure. Pick $f : X\to S$ and $g : Y \to S$ two $S$-schemes. Form the fiber product $r : X\times_S Y \to S$. From the exchange property applied to the cartesian square 
\[
\xymatrixcolsep{5pc}
\xymatrix{
X\times_S Y \ar[d]_{f'} \ar[r]^-{g'} \ar[rd]^-{r} &X \ar[d]^{f} \\
Y \ar[r]_-{g} &S
}
\]
and the projection formula, we have
\begin{eqnarray*}
\chi_S([X]\cdot[Y])&=&\chi_S([X\times_S Y])=[\Mot_{S,c}(X\times_S Y)]=[r_!\un_{X\times_S Y}]=[g_!f'_!g'^*f^*\un_S]\\
&=&[g_!g^*f_!f^*\un_S]=[g_!g^*\Mot_{S,c}(X)]=[g_!(g^*\Mot_{S,c}(X)\otimes_Y \un_Y)]\\
&=&[\Mot_{S,c}^\vee(X)\otimes_S g_!\un_Y)]=[\Mot_{S,c}^\vee(X)\otimes_S \Mot_{S,c}(Y)]\\
&=&\chi_S([X])\cdot \chi_S([Y]).
\end{eqnarray*}
\end{proof}

\begin{Proposition}
\label{prop-compactcohom-motive-smooth}
Let $f : X \to S$ a smooth morphism of pure relative dimension $d$. Then 
\[
[\Mot_{S,c}^\vee(X)]=[\Mot_S(X)(-d)]\in \K{\SH{S}}.
\]
\end{Proposition}
\begin{proof}
By definition, $\Mot_{S,c}^\vee(X)=f_!f^*(\un_S)$ and $f_!=f_*\Th^{-1}(\Omega_f)$. As $\Mot_{S,c}(-)$ is additive and $\Omega_f$ is locally free, we can assume $\Omega_f$ is free (of rank $d$). In that case, $\Th^{-1}(\Omega_f)=(-d)[-2d]$. The result now follows because suspension function is idempotent in the Grothendieck group. 
\end{proof}

\subsection{From $\K{\Varmu{k}}$ to $\K{\QUSH{k}}$}
\label{section-chihatmu}

Let $X=\Spec{A}$ a $k$-scheme of finite type, $r\in \Nn^*$ and$f\in A^\times$. We denote $Q_r^{gm}(X,f)$ the ${\mathbb{G}_{m}}_k$-scheme 
\[
\Spec{A[T,T^{-1},V]/(V^r-fT)}\to {\mathbb{G}_{m}}_k=\Spec{k[T,T^{-1}]}.
\]

More generally, we define by gluing for $X=\Spec{A}$ $k$-scheme of finite type, $r\in \Nn\backslash\set{0}$, $f\in\Gamma(X,\mathcal{O}_X^\times)$.

\[
Q_r^{gm}(X,f)\to {\mathbb{G}_{m}}_k.
\]

Let $\QUSH{k}$ the triangulated subcategory with infinite sums of $\SH{{\mathbb{G}_{m}}_k}$ spanned by objects $\Sus_T^p(Q_r^{gm}(X,f)\otimes A_{\mathrm{cst}})$ for $X$ smooth $k$-scheme. (define $A\in \mathcal{E}$). 

Let $q : {\mathbb{G}_{m}}_k \to \Spec{k}$ the structural projection and $1 : \Spec{k}\to {\mathbb{G}_{m}}_k$ its unit section.

 \begin{Proposition}
 \label{prop-chi-muhat}
There is a unique ring morphism 
\[
\chi_{\hat\mu} : \K{\Varmu{k}} \to \K{\QUSH{k}}
\]
 such that the class of $X[V]/(V^r-f)$, for $X$ a smooth $k$-scheme, $f\in \Gam(X,\mathcal{O}_X^\times)$, $r\in \Nn\backslash\set{0}$ (with the $\mu_n$ action on $V$), is send to $[\Mot_{\Gm_k,c}^\vee(Q_r^{gm}(X,f))]$.
 \end{Proposition}
 
 \begin{proof}
The ring morphism 
\[
\chi_{\hat\mu} : \K{\Varmu{k}} \to\K{\Var{\Gm_k}}\to \K{\SH{\Gm_k}}.
\]
is defined by composition of maps from Corollary \ref{cor-varmu-vargm} and Proposition  \ref{realisation-motives-is}. 

Hence it suffices to show that the image of this morphism lies in $\K{\QUSH{k}}$. From the proof of \ref{iso-res-varmu}, $\K{\Varmu{k}}$ is generated by classes of $X[V,V^{-1}]/V^r=f$ as in the statement of the proposition. Hence it suffices to show that 
\[
[\Mot_{\Gm_k,c}(Q_r^{gm}(X,f))]\in \K{\QUSH{k}}.
\] 
But  $\QUSH{k}$ is the triangulated subcategory with infinite sums generated by the set of objects $\Sus_T^p(Q_r^{gm}(X,f)\otimes A_{\mathrm{cst}})$, which is stable by Tate twist, hence by Proposition \ref{prop-compactcohom-motive-smooth}, $[\Mot_{\Gm_k,c}(Q_r^{gm}(X,f))]\in \K{\QUSH{k}}$. 
 \end{proof}

\begin{Lemma}
\label{lem-com-diag-varmu-var-QUSH-SH}
We have a commutative diagram :
\[
\xymatrixcolsep{5pc}
\xymatrix{
 \K{\Varmu{k}} \ar[d]_{\chi_{\hat\mu} } \ar[r]^-{} & \K{\Var{k}} \ar[d]^{\chi_k} \\
\K{\QUSH{k}} \ar[r]_-{1^*} &\K{\SH{k}},
}
\]
where $ \K{\Varmu{k}} \to  \K{\Var{k}}$ is induced by the forgetful functor and 
\[
1^* :\K{\QUSH{k}} \to \K{\SH{k}}
\]
is the composite 
\[
\K{\QUSH{k}} \hookrightarrow \K{\SH{{\mathbb{G}_{m}}_k}} \overset{1^*}{\longrightarrow} \K{\SH{k}}.
\]
\end{Lemma}

\begin{proof}
It suffices to show that the following diagram is commutative :
\[
\xymatrixcolsep{5pc}
\xymatrix{
 \K{\Var{\Gm_k}} \ar[d]_{\chi_{\Gm_k} } \ar[r]^-{} & \K{\Var{k}} \ar[d]^{\chi_k} \\
\K{\SH{\Gm_k}} \ar[r]_-{1^*} &\K{\SH{k}},
}
\]
where the upper map is induced by picking the fiber above 1 of a $\Gm_k$-variety. For $X\in \Var{\Gm_k}$, we consider the following cartesian square :
\[
\xymatrixcolsep{5pc}
\xymatrix{
 X' \ar[d]_{f'}  \ar[r]^-{1'} & X \ar[d]^{f} \\
k \ar[r]_-{1} &{\Gm_k}.
}
\]
One needs to show that $1^*\Mot_{\Gm_k,c}(X)\simeq  \Mot_{k,c}(X')$. By \cite[Scholie 1.4.3]{ayoub_six_1}, there is a 2-isomorphism $f'_!1'^*\cong 1^*f_!$. Hence
 \[1^*\Mot_{\Gm_k,c}(X)=1^*f_!f^*\un_{\Gm_k}\simeq f'_!1'^*f^*\un_{\Gm_k}\simeq f'_!f'^*1^*\un_{\Gm_k}=\Mot_{k,c}(X'). \]
\end{proof}

\subsection{Rigid analytic geometry}
\label{section-rigid-geometry}

We will use the formalism of Tate's rigid analytic geometry \cite{tate_rigid_1971}. For details and proofs, see \cite{bosch_non-archimedean_1984} and \cite{fresnel_rigid_2004}.

Denote by 
\[
\abs{\cdot} : K\to \Rr
\]
the ultrametric absolute value induced by the valuation on $K=k((t))$. 

Let $(A,\abs{\cdot})$ be a complete ring equipped with a semi-norm. Define the ring of converging power series by 
\[
A\{T_1,...,T_n\}:=\set{\sum_{i=(i_1,...,i_n)\in \Nn^n} a_iT_1^{i_1}...T_n^{i_n}\mid \underset{i\to +\infty}{\abs{a_i}}\to 0 }.
\]
Recall that an affinoid $K$-algebra is a $K$-algebra $A$ such that there is a  surjective morphism $ \pi : K\{T_1,...,T_n\}\to A$.

To such an affinoid $K$-algebra $A$ one can associate a affinoid rigid analytic $K$-variety $\Spm{A}$ with underlying set the set of maximal ideals of $A$. 

Let us introduce the notation
\[
A<f_1\mid f_2,...,f_n>=A\{T_2,...,T_n\}/(f_iT_i=f_1)_{i=2,...,n},
\]
for $f_1,...,f_n\in A$ generating the unit ideal.

If $X=\Spm{A}$ is a $K$-affinoid space, $f,g\in A$ and, $p,q\in \Nn^*$ such that $\abs{f}^{1/p}\leq \abs{g}^{1/q}$, let 
$\Bb_X(o,\abs{f^{1/p}})=\Spm{A\{f^{-1/p}T\}}$ where
\[
A\{f^{-1/p}T\}=\set{h=\sum_i a_i t^i\in A[[t]]\mid \lim_i \abs{a_i^pf^i}=0}.
\]
We define similarly annuli $\Ccr_X(o,\abs{f}^{1/p},\abs{g}^{1/q})=\Spm{A\{g^{-1/q}T\}<T^p\mid f>}$ and thin annuli
$\parBb_X(o,\abs{f}^{1/p})=\Ccr_X(o,\abs{f}^{1/p},\abs{f}^{1/p})$. 

The variety $\Bb_X(o,\abs{f^{1/p}})$ represents a variety of closed balls parametrized by $X$ and of radii $\abs{f}^{1/p}$. Similarly, $\Ccr_X(o,\abs{f}^{1/p},\abs{g}^{1/q})$ is a family of closed annuli  of radii between $\abs{f}^{1/p}$ and $\abs{g}^{1/q}$. Such constructions can be glued in order to define similar notions for any rigid variety $X$. 

If $X=\Spm{K}$, we will drop the $X$ to ease the notation. If $p/q\in \Qq$ and $r=\abs{t^p}^{1/q}$ we will denote $\Bb(o,r)=\Bb(o,\abs{t^p}^{1/q})$ the closed ball of radius $r$. In particular, $\cup_{r<1} \Bb(o,r)$ is the unit open ball. Denote $\Aa^{1,\an}_K=\cup_{r>0} \Bb(o,r)$ the rigid affine line. Observe that neither the unit open ball or $\Aa^{1,\an}_K$ are quasi-compact.

Recall also the analytfication functor
\[(-)^{\an}: \Sch_A \to \VarRig_A,\]
such that 
\[\Hom_\Spec{A}(\Spec{B},X)\simeq\Hom_\Spm{A}(\Spm{A},X^\an)\]
for $X$ a separated scheme. It preserves smoothness, properness. 

\subsection{Formal schemes}
\label{section-formal-schemes}

We will use Raynaud's viewpoint on formal schemes in order to make a bridge between algebraic varieties over $k$ and rigid varieties over $K=k((t))$. See \cite{berthelot_cohomologie_1996} or \cite{le_stum_rigid_2007} for details on tubes. Denote by $R=k[[t]]$ the valuation ring of $K$. 

All the formal $R$-schemes we consider are assumed to be topologically of finite type. 

Given an affine formal $R$-scheme $\calX=\Spf{A}$, the $K$-algebra $K\otimes_R A$ is affinoid. Denote by $\calX_\eta$ its maximal spectrum. It is the generic fiber of $\calX$. By gluing affine pieces, such a construction gives rise to the generic fiber functor
\[
(-)_\eta : \mathrm{FSch}_R\to \VarRig_K.
\]
If $X$ is a rigid $K$-variety, a formal model of $X$ is a formal $R$-scheme $\calX$ together with an isomorphism $X\simeq \calX_\eta$. 

\begin{Proposition}
Any separated quasi-compact rigid $K$-variety admits an admissible formal $R$-model. 
\end{Proposition}

We also have a functor 
\[
(-)_\sigma : \mathrm{FSch}_R\to \Var{k}
\]
induced by $X\mapsto X\times_R k$, it is the special fiber of $X$ and can be though as the reduction modulo $t$ of $X$. At the level of topological spaces, $X$ and $X_\sigma$ are in bijection, because $\Spf{A}\simeq \Spec{A/(t)}$ for any $R$-algebra $A$ topologically of finite type. 

\begin{Proposition}
Given an admissible formal $R$-scheme $\calX$, there is a canonical map, called the specialization map (or the reduction map),  
\[
\spc : \calX_\eta \to \calX_\sigma.
\]
It is defined at the level of topological spaces and is surjective on the closed points of $\calX_\sigma$.  
\end{Proposition}

\begin{Proposition}
\label{prop-openim-genericfiber}
Given an $R$-scheme $X$, we can consider its formal $t$-adic completion $\calX$. We can also base change to $K$ and take the analytification. Then there is an open immersion of rigid varieties
\[
\calX_\eta \to (X\times_R K)^\an.
\]
Assuming $X$ is flat and irreducible, this immersion is an isomorphism if and only if $X$ is proper on $R$. 
\end{Proposition}

\begin{Definition}
Let $\calX$ a formal $R$-scheme. 
If $D$ is a locally closed subset of the special fiber, the tube of $D$ in $\calX$ is the inverse image $]D[_\calX:=\spc^{-1}(D)$, with its reduced rigid variety structure. It is an open rigid analytic subvariety of $\calX_\eta$. When there is no possible confusion, we will denote $]D[=]D[_\calX$. 
\end{Definition}

If $\calU$ is an open formal subscheme of $\calX$ such that $D\subset \calU_\sigma$, then $]D[_\calX =]D[_\calU$. In particular, $\tube{\calU_\sigma}_\calX=\calU_\eta$.

\begin{Proposition}[{\cite[Proposition 1.1.14]{berthelot_cohomologie_1996}}]
\label{prop-tube-admissible-cover}
Let $\calX$ a formal $R$-scheme, locally of finite type, and $D\subset \calX_\sigma$. 

\begin{enumerate}
\item For any open cover $(V_i)_i$ of $V$, the tubes $\tube{V_i}$ form an admissible cover of $\tube{D}$. 
\item If $(D_i)_i$ is a finite closed cover of $D$, then the tubes $\tube{D_i}$ form an admissible cover of $\tube{D}$. 
\end{enumerate}
\end{Proposition}

\begin{Proposition}[{\cite[Proposition 1.3.1]{berthelot_cohomologie_1996}}]
\label{prop-berth-iso-tubes-etale}
Let $D$ a $k$-scheme of finite type, $\calX$, $\calX'$ two formal $R$-schemes of finite type, two immersions $i : D \hookrightarrow \calX_\sigma$, $i' : D \hookrightarrow \calX'_\sigma$, $u : \calX'\to \calX$ an \'etale morphism such that $i=i'\circ u$. Then
\[u_K : \tube{D}_{\calX'}\to \tube{D}_\calX\]
is an isomorphism. 
\end{Proposition}

\begin{Definition}
Let $\calX$ a formal $R$-scheme of finite type. Say that $\calX$ is semi-stable if for every $x\in \calX_\sigma$, there is a regular open formal subscheme $\calU\subset \calX$ containing $x$ and elements $u,t_1,...,t_r\in \calO(\calU)$ such that the following properties hold:
\begin{enumerate}
\item $u$ is invertible and there are positive integers $N_1,...,N_r$ such that $t=ut_1^{N_1}\cdots t_r^{N_r}$,
\item for every non empty $I\subset \set{1,...,r}$, the subscheme $D_I\subseteq \calU_\sigma$ defined by equations $t_i=0$ for $i\in I$ is smooth over $k$, has codimension $\abs{I}-1$ in $\calU_\sigma$ and contains $x$. 
\end{enumerate}

\end{Definition}
If $\calX$ is a formal $R$-scheme, $f\in \Gam(\calX, \calO_X)$, $\underline{N}\in \Nn^k$, we define
 \[
 \St^{f}_{\calX,\underline{N}}=\calX\{T_1,...,T_k\}/(T_1^{N_1}...T_k^{N_k}-f).
 \]

\begin{Proposition}[{\cite[Proposition 1.1.60]{ayoub_rigide}}]
\label{prop-semistable-etaleproj}
Let $X$ a semi-stable $R$-scheme of type $\underline{N}$ around some $x\in X$. Then for some Zariski neighborhood $U$of $x[V,V^{-1}]$ in $X[V,V^{-1}]$, there is an \'etale $R$-morphism 
\[ U\to \St^{Ut}_{\Spf{R[U,U^{-1},S_1,...,S_r}],\underline{N}}.\]
\end{Proposition}

\subsection{Rigid motives}
\label{section-rigid-motives}
Ayoub builds in \cite{ayoub_rigide} a category $\RigSH{K}$ of rigid motives over $K$, in an analogous manner of $\SH{K}$, but with starting point rigid varieties instead of schemes. 

As in the algebraic case, one needs to choose a category of coefficients $\mathfrak{M}$, the main examples being $\mathrm{RigSH}(K)$ and $\RigDA{K,\Lambda}$. 

We need to use a topology on rigid varieties, however for technical reasons, it is not the straightforward translation of Nisnevich topology that will be used.

\begin{Definition}
\begin{itemize}
\item First we define the notion of Nisnevich cover of formal schemes. A family $\calR=(u_i :\calU_i\to \calX)_{i\in I}$ is a Nisnevich cover of a formal scheme $\calX$ if the $u_i$ are \'etale morphisms and the reduction $\calR_\sigma=(u_{i \sigma} : \calU_{i \sigma} \to \calX_\sigma)_{i\in I}$ is a Nisnevich cover of $\calX_\sigma$ (in the algebraic sense). 

\item Let $X$ a quasi-compact smooth rigid $K$-variety and $\calX$ a formal $R$-model of $X$. Then a family $(U_i\to X)_{i\in I}$ is a Nisnevich cover of $X$ if there is a formal $R$-model $\calX$ of $X$ and a Nisnevich cover $\calR$ of $\calX$ such that $\calR_\eta$ is a refinement of $(U_i\to X)_{i\in I}$. 

\item A Nisnevich cover of a (non-quasi-compact) rigid variety $X$ is a family of \'etale morphisms $(U_i\to X)_{i\in I}$ such that for any rigid quasi-compact smooth rigid $X$-variety $V\to X$, the family $(U_i \times_X V\to V)_{i\in I}$ is a Nisnevich cover of $V$. 
\end{itemize}
\end{Definition}

An admissible cover of some rigid variety $X$ is in particular a Nisnevich cover.

The other ingredient we need to change in the construction of $\RigDA{K,\Lambda}$ is what we want to localize. We replace $\Aa^1$-invariance by $\Bb^1$-invariance, and $T$-spectra by $T^{\an}$-spectra, where $T^{\an}$ is the cokernel of the inclusion (of sheaves represented by) $\parBb(o,1)\to \Bb(o,1)$.
The category $\RigDA{K,\Lambda}$ is now defined as the Verdier localization of the derived category of $T^{\an}$-spectra of $\Lambda$-sheaves of smooth rigid varieties for the Nisnevich topology, with respect to the subcategory generated by complexes
\[
[\dots\to 0\to \Sus_{T^\an}^r(\Bb_U(o,1))\to \Sus_{T^\an}^r(U)\to 0\to\dots]
\]
and 
\[
[\dots\to 0\to \Sus_{T^\an}^r(U\otimes T^\an)\to \Sus_{T^\an}^{r+1}(U)\to 0\to\dots].
\]

As in the algebraic case, we have suspension functor $\Sus_{T^\an}^r(-)$, the tensor $-\otimes A$ has a right adjoint $\intHom(-,A)$, the Tate twist $(-)(n)$ is also defined as in the algebraic case. 

\begin{Definition}
We define for $X$ a smooth rigid $K$-variety its homological motive by $\Motr(X)=\Sus_{T^\an}^0(X\otimes \un_K)$ and dually, its cohomological motive by $\Motr^\vee(X)=\intHom(\Motr(X),\un_K)$.
\end{Definition}

A relative version $\RigSH{S}$ of motives of rigid $S$-varieties exists and satisfies various fonctorialities, like the existence of a pair of adjoint functors $(f^*,f_*)$ for any $f : S'\to S$. To our knowledge a full six functor formalism is not available in this context, the missing ingredients being $f_!$ and $f^!$. Hence there is no already defined notion of compactly supported rigid motive. 

The analytification functor induce a (monoidal triangulated) functor 
\[
\Rig^* : \SH{K}\to \RigSH{K}.
\]
Such a functor is compatible in a strong sense with the six operations defined on $\SH{-}$, see \cite[Th\' eor\`eme 1.4.40]{ayoub_rigide}, but we do not need it.

The first point of the following proposition is direct consequence of $\Bb^1$-invariance imposed in the construction of $\RigSH{K}$, the others follows after some reductions. 
\begin{Proposition}[{\cite[Proposition 1.3.4]{ayoub_rigide}}]
\label{prop-wequiv-annuli}
Let $X$ a rigid $K$-variety, $f,g,h\in \Gam(X,\mathcal{O}_X^\times)$ and $p, q, r\in \Nn\backslash\set{0}$ such that $\abs{f(x)}^{1/p}\leq \abs{g(x)}^{1/q}\leq \abs{h(x)}^{1/r}$. Then we have the following isomorphisms in $\RigSH{K}$ :
\begin{itemize}
\item $\Motr(\Bb_X(o,f^{1/p}))\simeq \Motr(X)$,
\item$ \Motr(\Bb_X(o,f^{1/p}))\simeq \Motr(X)$,
\item$ \Motr((\Aa^1_K)^{\an} \times_K X) \simeq \Motr(X)$,
\item $\Motr(\partial \Bb_X(o,g^{1/q})) \simeq \Motr(\Ccr_X(o,f^{1/p},h^{1/r}))$,
\item $\Motr(\partial \Bb_X(o,g^{1/q}))\simeq \Motr((\Gm_{K}^{\an} \times_K X)) $,
\item $\Motr(\partial \Bb_X(o,g^{1/q}) \simeq \Motr((\Bb_X(o,g^{1/q})\backslash \set{o_X})) $.
\end{itemize}

\end{Proposition}

Let $X$ a smooth $k$-scheme, $f\in \Gamma(X,\mathcal{O}_X^\times)$ and $p\in \Nn^*$. Then define $Q^\mathrm{for}_p(X,f)$ as the $t$-adic completion of the $R$-scheme $X\times_k R[V]/(V^p-tf)$, and $Q^\Rig_p(X,f)$ the generic fiber of $Q^\mathrm{for}_p(X,f)$. 

Define also $Q^\an_p(X,f)$ as the analytification of $X\times_k K[V]/(V^p-tf)$. By Proposition \ref{prop-openim-genericfiber}, there is an open immersion of rigid $K$-varieties 
\[
Q^\Rig_p(X,f)\to Q^\an_p(X,f).
\]

\begin{Theorem}[{\cite[Th\'eor\`eme 1.3.11]{ayoub_rigide}}]
\label{thm-ayoub-equi-Qrig-Qan}
Let $X$ a smooth $k$-scheme, $f\in \Gamma(X,\mathcal{O}_X^\times)$, and $p$ a positive integer. Then the inclusion
\[
Q_p^\Rig(X,f)\to Q_p^\an(X,f)
\]
induce an isomorphism
\[
\Motr(Q_p^\Rig(X,f))\simeq \Motr(Q_p^\an(X,f)).
\]
\end{Theorem}

Define a functor $\mathfrak{F} : \QUSH{k}\to \RigSH{K}$ as the composite
\[
\mathfrak{F} : \QUSH{k}\to \SH{\Gm_k}\underset{\pi^*}{\to}\SH{K}\underset{\Rig^*}{\to} \RigSH{K},
\]
where $\pi : \Spec{K}\to \Gm_k$ corresponds to the ring morphism $k[T,T^{-1}]\to K=k((t))$ sending $T$ to $t$. Observe that $\mathfrak{F}$ send the generators $\Mot_{\Gm_k}(Q^\gm_p(X,f))\in\QUSH{k}$ to $\Motr(Q^\an_p(X,f))$. 

One of the main results of Ayoub in \cite{ayoub_rigide} is the following theorem. 
\begin{Theorem}[{\cite[Scholie 1.3.26]{ayoub_rigide}}]
\label{thm-scholie-Rig-QUSH} 
The functor
\[
\mathfrak{F} : \QUSH{k}\to \RigSH{K}
\]
is an equivalence of categories, denote by $\mathfrak{R}$ a quasi-inverse.
\end{Theorem}

\section{Realization map for definable sets}
\label{section-realization-map}

This aim of this section is to define a morphism $\chir : \K{\VF_K}\to \K{\RigSH{K}}$. We will first define it on $\K{\Gam[*]}$ and $\K{\RES_K[*]}$ in Sections \ref{section-gam-part} and \ref{section-res-part}, before using Hrushovski and Kazhdan isomorphism of define it on $\K{\VF_K}$ in Section \ref{section-def-chirRV}. We will in fact define two morphisms $\chir$ and $\chirp$. Section \ref{section-motive-tube} is devoted to the proof of Theorem \ref{main-thm1} via the study of motives of tubes in a semi-stable situation, the main results are grouped in Section \ref{section-compatib-chir}. The last Section \ref{section-few-more-realization} is devoted to the definition of two other realization maps $\K{\VF_K}\to \K{\RigSH{K}}$ and the statement of analogous of Theorem \ref{main-thm1} for them.

\subsection{The $\Gam$ part}
\label{section-gam-part}

Recall that $\Gamn{n}$ is the category with objects subsets of $\Gam^n$ defined by piecewise linear equations and inequations with $\Zz$-coefficients. A morphism between $Y$ and $Y'$ is a bijection defined piecewise as $x\mapsto Ax+B$

 with $A\in GL_n(\Zz)$ and $B\in \Zz^n$.  From this one forms the graded
\[
\Gamn{*}:= \bigoplus_{n\in \Nn} \Gamn{n}. 
\]
One defines also $\Gamnfin{n}$ and $\Gamnfin{*}$ to be the full subcategories of $\Gamn{n}$ and $\Gamn{*}$ whose objects are finite. 

Note that morphisms in $\Gamn{*}$ must respect the graduation, in particular, the class $[\set{0}]_1\in \Gamn{1}$ is not equal to the class $[\set{0}]_n\in \Gamn{n}$ for $n\neq 1$. Recall the $o$-minimal Euler characteristics $\eu$ and $\eu_c$ defined above Proposition \ref{prop-def-calE}.

\begin{Definition}
If $[X]\in \K{\Gamn{*}}$, with $X\in \Gamn{d}$, define 
\[
\chirGam(X)=\eu_c(X)[\Motr(\parBb(o,1)^d)(-d)]\in \K{\RigSH{K}}
\]
and 
\[
\chirpGam([X])=\eu(X)[\Motr(\parBb(o,1)^d)]\in \K{\RigSH{K}}.
\]
Hence we get two ring morphisms
\[
\chirGam, \chirpGam : \K{\Gamn{*}} \to \K{\RigSH{K}}.
\]

\end{Definition}
It is well defined because $\K{\Gamn{*}}$ is naturally graded and by additivity of Euler characteristic.

\begin{Proposition}
\label{prop-chirGam-locclosed}
Let $X\subseteq \Gam^n$ be a convex compact polytope.  If $X$ is closed, then
\[
\chirGam([X])=[\Motr(\val^{-1}(X)^\Rig)(-n)] \;\mathrm{and}\;\chirpGam([X])=[\Motr(\val^{-1}(X)^\Rig)].
\]
If $X$ is open, then
\[
\chirGam([X])=(-1)^n[\Motr(\val^{-1}(X)^\Rig)(-n)]  \;\mathrm{and}\;
\chirGam([X])=(-1)^n[\Motr(\val^{-1}(X)^\Rig)].
\]
\end{Proposition}
\begin{proof}
If $X$ is empty, $\eu(X)=\eu_c(X)=0$ hence the proposition is verified. Hence we can suppose $X$ in nonempty. We have $\eu(X)=\eu_c(X)=1$ if $X$ is closed, and $\eu(X)=\eu_c(X)=(-1)^n$ if $X$ is open. Hence the result follows from the following Lemma \ref{lem-Mrig-convex-polyhedron}.
\end{proof}

\begin{Lemma}
\label{lem-Mrig-convex-polyhedron}
Let $X\subseteq \Gam^n$ be a non-empty convex polytope, either closed or open. Then 
\[
\Motr(\val^{-1}(X)^\Rig)\simeq \Motr(\parBb(o,1)^n).
\]
\end{Lemma}

\begin{proof}
We first assume that $X$ is closed. 
We work by induction on $n$. If $n=1$, then 
\[X=\set{x\mid \alpha \leq px\leq \beta}\] for $\alpha, \beta \in \Zz, p\in\Nn$. Hence
\[
\val^{-1}(X)^\Rig=\Ccr(o,\abs{\pi^\beta}^{1/p}, \abs{\pi^\alpha}^{1/p}).
\]
By Proposition \ref{prop-wequiv-annuli}, $\Motr(\val^{-1}(X)^\Rig)=\Motr(\parBb(o,1))$. 

Suppose now the result is known for $n-1$. There are finitely many affine functions $h_i$, $i\in I_0$ with $\Zz$-coefficients such that 
\[
X=\set{x\in \Gam^n\mid h_i(x)\geq 0}.
\]

We can rewrite $X$ as
\[
X=\set{(x,y)\in \Gam\times \Gam^{n-1}\mid p_ix\leq f_i(y), q_jx \geq g_j(y), i\in I, j\in J}
\]
for some (possibly empty) finite sets $I,J$ integers $p_i,q_j\in \Nn$, affine functions $f_i, g_j: \Gam^{n-1}\to \Gam$ with $\Zz$ coefficients. Now observe that the projection of $X$ on the last $n$-th coordinates is 
\[
Y=\set{y\in \Gam^{n-1}\mid \forall (i,j)\in I\times J, p_i g_j(y)\leq q_j f_i(y)}.
\]
It satisfies the hypotheses of the proposition hence we get that $[\Motr(\val^{-1}(Y)^\Rig)]=[\Motr(\parBb(o,1)^{n-1})]$ by induction. 

We set $I'=\set{i\in I\mid p_i>0}$ and $J'=\set{j\in J\mid q_j>0}$ and observe that
\[
X=\set{(x,y)\in \Gam\times Y\mid p_ix\leq f_i(y), q_jx \geq g_j(y), i\in I', j\in J'}.
\]
Set $X_{i,j}=\Ccr_{Y^\Rig}(o,\abs{\tilde{f_i}^{1/p_i}},\abs{\tilde{g_j}^{1/q_j}})$, where we used the notation that if
 \[
 f(y_1,... ,y_{n-1})=b+a_1y_1+...+a_{n-1}y_{n-1},
 \] then 
 \[
\tilde{f}(x_1,... ,x_{n-1})= t^b\cdot x_1^{a_1}\cdot ... \cdot x_{n-1}^{a_{n-1}}.
\]
We have now 
\[X^\Rig=\bigcap_{(i,j)\in I'\times J'} X_{i,j}.\]

Set 
\[
Y_{i,j}=\set{y\in Y^\Rig\mid\forall (i',j')\in I\times J, \abs{\tilde{f_i}^{1/p_i}}\geq \abs{\tilde{f_{i'}}^{1/p_{i'}}}, \abs{\tilde{g_j}^{1/q_j}}\leq \abs{\tilde{g_{j'}}^{1/q_{j'}}}}.
\]
The $(Y_{i,j})_{(i,j)\in I\times J}$ form an admissible cover of $Y^\Rig$, indeed, $Y_{i,j}$ is defined in $Y^\Rig$ by some non-strict valuatives inequalities, if $D$ is a rational domain, the standard cover of $D$ induced by functions used to defined $D$ and the functions used to defines all $Y_{i,j}$ gives the required refinement of $(D\cap X_{i,j})_{(i,j)\in I\times J}$. 

Then it suffices to show the result for $X^\Rig \cap (Y_{i,j}\times_K\Aa_K^{1,\an})$. But we have then 
\[
X^\Rig \cap (Y_{i,j}\times_K\Aa_K^{1,\an})=\Ccr_{Y_{i,j}}(o,\abs{\tilde{f_i}^{1/p_i}},\abs{\tilde{g_j}^{1/q_j}})
\]
hence the result follows from Proposition \ref{prop-wequiv-annuli},  which gives 
\[
\Motr(\Ccr_{Y_{i,j}}(o,\abs{\tilde{f_i}^{1/p_i}},\abs{\tilde{g_j}^{1/q_j}}))\simeq\Motr(\parBb_{Y_{i,j}}(o,1)).
\]

Suppose now that $X$ is an open polyhedron. There are finitely many affine functions $h_i$, $i\in I_0$ with $\Zz$-coefficients such that 
\[
X=\set{x\in \Gam^n\mid h_i(x)> 0}.
\]
For $r\in \Qq_+^*$, set 
\[
X_r=\set{x\in \Gam^n\mid h_i(x) \geq r}.
\]
As $X$ is non-empty, $X_r$ is non-empty for $r$ small enough and  by the previous case, $\val^{-1}(X_r)^\Rig\simeq \parBb(o,1)^n$. 
As $\val^{-1}(X)^\Rig=\cup_{r\in \Qq_+^*}\val^{-1}(X_r)^\Rig$, we have the result. 
\end{proof}

We will not use it, but Proposition \ref{prop-chirGam-locclosed} can be extended to all closed polyhedral complexes. 
\begin{Proposition}
Let $X\subseteq \Gamma^n$ be (the realization of) a closed polyhedral complex. Then 
\[
\chirGam([X])=\eu(X)[\Motr(\val^{-1}(X)^\Rig)].
\]
\end{Proposition}

\begin{proof}
We work by double induction on the maximal dimension of simplex in $X$ and number of simplex of maximal dimension. 
Let $\Delta\subset X$ a simplex of maximal dimension. Let $Y=X\backslash \Delta^\circ$, with $\Delta^\circ$ the interior of $\Delta$, $\partial\Delta=\Delta\backslash\Delta^\circ$. 

Then $(\val^{-1}(\Delta)^\Rig,\val^{-1}(Y)^\Rig)$ is an admissible cover of $\val^{-1}(X)^\Rig$, with intersection $\val^{-1}(\partial\Delta)^\Rig$ hence
\[
[\Motr(\val^{-1}(X)^\Rig)]=[\Motr(\val^{-1}(\Delta)^\Rig)]+[\Motr(\val^{-1}(Y)^\Rig)]-[\Motr(\val^{-1}(\partial\Delta)^\Rig)].
\]
By Lemma \ref{lem-Mrig-convex-polyhedron}, $[\Motr(\val^{-1}(\Delta)^\Rig)]=[\parBb(o,1)^n]$. Apply the induction hypothesis to get $[\Motr(\val^{-1}(\partial\Delta)^\Rig)]=(1-(-1)^{d})([\parBb(o,1)^n]$ and $[\Motr(\val^{-1}(Y)^\Rig)]=\eu(Y)[\parBb(o,1)^n]$.
We have the result, since $\eu(X)=\eu(Y)+(-1)^d$.
\end{proof}

\subsection{The $\RES$ part}
\label{section-res-part}

Recall that the category $\RES_K[*]$ is defined as the full subcategory of $\RV_K[*]$ such that for $X\in \RV_K[*]$, $X\in \RES_K[*]$ if and only if $\val(X)$ is finite. 
Alternatively, one can define $\RES_K$ as the category of definable sets consisting of $V_\gamma$ for $\gamma\in \Qq$, the residue field $\kk$ together with the functions $H : V_{\overline{\gamma}} \to \kk$ associated to every $\overline{\gamma}$-polynomial $H$. 

Fix $\overline{\gamma}= (\gamma_1,...,\gamma_n)\in \Zz^n$. A $\overline{\gamma}$-weighted monomial is a term $a_\alpha X^\alpha=a_\alpha \Pi_i X_i^{\alpha_i}$ with $a_\alpha\in \RV(K)$, $\alpha_i\in \Nn$ such that $\val(a_\alpha)+\sum_i \alpha_i \gamma_i=0$. 
A $\overline{\gamma}$-polynomial $H$ is a finite sum of $\overline{\gamma}$-monomials. It leads to a well defined function $V_{\overline{\gamma}} \to \kk$.

\begin{Definition}
Define ring morphisms
\[
\chirRes : \K{\RES_K}[*] \to \K{\RigSH{K}}
\]
and 
\[
\chirpRes : \K{\RES_K}[*] \to \K{\RigSH{K}}
\]
by the formulas, for $X$ a smooth $k$-variety of pure dimension~$r$, $f\in \Gam(X,\mathcal{O}^\times_X)$, $m\in \Nn^*$, 
\[\chirRes([Q^\RV_m(X,f)]_n)=[Q^\Rig_m(X,f)(-r)]\]
and
\[\chirpRes([Q^\RV_m(X,f)]_n)=[Q^\Rig_m(X,f)(n-r)].\]
\end{Definition}
As the $[Q^\RV_m(X,f)]_n$ generate $ \K{\RES_K[*]}$ by Corollary \ref{cor-QRES-generators-KRES}, one only need to show that the maps are well defined. But we can check that they coincide with the composite

\begin{multline*}
\K{\RES_K[*]}\to !\K{\RES_K}[\eL^{-1}]\overset{\Theta}{\to} \K{\Varmu{k}}[\eL^{-1}]\overset{\chi_{\hat\mu}}{\to} \\
\K{\QUSH{k}}\overset{\mathfrak{F}}{\to} \K{\RigSH{K}}.
\end{multline*}

where the map $\K{\RES_K[*]}\to !\K{\RES_K}[\Aa^1_k]^{-1}$ is $[X]_n\mapsto [X]$ for $\chirRes$ and $[X]_n\mapsto [X]\eL^{-n}$ for $\chirpRes$. The maps $\Theta$, $\chi_{\hat\mu}$, $\mathfrak{F}$ are respectively defined at Propositions \ref{iso-res-varmu}, \ref{prop-chi-muhat} and Theorem \ref{thm-scholie-Rig-QUSH}. Note that this also implies that it is a morphism of rings.

\subsection{Definition of $\chirRV$}
\label{section-def-chirRV}

Recall the isomorphism $\K{\RES_K[*]} \otimes_{\K{\Gamnfin{*}}} \K{\Gamn{*}}\to \K{\RV_K[*]}$ from Proposition \ref{prop-iso-RV-RES-Gam}. Using this isomorphism, to define a ring morphism $\K{\RV_K[*]} \to \K{\RigSH{K}}$, it suffices to specify two ring morphisms $\K{\RES_K[*]}\to \K{\RigSH{K}}$ and $\K{\Gamn{*}}\to \K{\RigSH{K}}$ that coincide on $\K{\Gamnfin{*}}$. 
\begin{Definition}
\label{def-chirRV}
Define
\[
\chirRV : \K{\RV_K[*]} \to \K{\RigSH{K}}
\]
using the morphisms $\chirRes$ and $\chirGam$ and
\[
\chirpRV : \K{\RV_K[*]} \to \K{\RigSH{K}}
\]
using the morphisms $\chirpRes$ and $\chirpGam$.
\end{Definition}
To show that is is well defined, one needs to check that if $A\subseteq \Gam^n$ is definable and finite, then $\chirGam([A])=\chirRes([\val^{-1}(A)]_n)$ and $\chirpGam([A])=\chirpRes([\val^{-1}(A)]_n)$.
By additivity, one can assume $A=\set{\alpha}$. Hence it follows from the following lemma.
\begin{Lemma}
Let $\alpha=(\alpha_1,...,\alpha_n)\in \Gam^n$ be definable. Then 
\[
\chirRes([\val^{-1}(\set{\alpha})]=[\parBb(o,1)^n(-n)]
\]
and
\[
\chirpRes([\val^{-1}(\set{\alpha})]=[\parBb(o,1)^n]
\]
\end{Lemma}

\begin{proof}
We have $\val^{-1}(\set{\alpha})=V^*_{\alpha_1}\times...\times V^*_{\alpha_n}$. Because $\chirRes$ and $\chirpRes$ are ring morphisms and $[\parBb(o,1)^n(-n)]=[\parBb(o,1)(-1)]^n$, we can assume $n=1$. Suppose $\alpha=k/m$, with $k\in \Zz$ and $m\in \Nn^*$ relatively prime. Let $a,b\in \Zz$ such that $am+bk=1$. In this case, we have
\[
V^*_{k/m}\simeq \set{(z,u)\in V^*_{1/m}\times V^*_0\mid z^m=\bt u^b}=Q_m^\RV(\Gm_k, u^b)
\]
via the isomorphism $w\in V^*_{k/m}\mapsto (\bt^aw^b,\bt^{-k}w^m)$.
But now, $Q_m^\Rig(\Gm_k, u^b)\simeq \parBb(o,k/m)$ via the isomorphism $(z,u)\mapsto z^k u^a$ and we know that $\Motr(\parBb(o,k/m))\simeq \Motr(\parBb(o,1))$ by Proposition \ref{prop-wequiv-annuli}.
\end{proof}

\begin{Remark}
If $X\subseteq \RV^n$ is definable, then $\chir([X]_m)=\chir([X]_n)$ for any $m\geq n$, hence $\chir$ does not depends on the grading in $\K{\RV[*]}$, it is in fact defined on $\K{\RV}$
\end{Remark}

\begin{Proposition}
The ring morphisms $\chirRV$ and $\chirpRV$ of Definition \ref{def-chirRV} induce ring morphisms 
\[
\chirRV, \chirpRV : \K{\RV_K[*]}/\Isp \to \K{\RigSH{K}}.
\]
\end{Proposition}

\begin{proof}
We need to check that the generator of $\Isp$ vanish under $\chirRV$ and $\chirpRV$. 

We have $\Motr(\set{1})=\Motr(\Spm{K})=\un_K$, hence $\chirRV([\set{1}]_1)=\chirRV([\set{1}]_0)=\chirpRV([\set{1}]_0)=[\un_K]$ and $\chirpRV([\set{1}]_1)=[\un_K(1)]$. 

Moreover, $\RV^{>0}=\val^{-1}((0,+\infty))$ and $\eu_c((0,+\infty))=0$, $\eu((0,+\infty))=-1$. Hence $\chirRV([\RV^{>0}]_1)=0$ and $\chirRV([\RV^{>0}]_1)=-[\Motr(\parBb(o,1))]$. 

We already see that 
\[\chirRV([\set{1}]_1)=\chirRV([\set{1}]_0+[\RV^{>0}]_1).\]

For $\chirpRV$, we need more. As in the proof of Proposition \ref{realisation-motives-is}, there is an distinguished triangle
\[
\Mot_c(\Gm_k)\to \Mot_c(\Aa_k^1)\to \Mot_c(\Spec{k})\to\Mot_c(\Gm_k)[1].
\]
As $\Mot_c(\Gm_k)=\Mot(\Gm_k)(-1)[-2]$, $\Mot_c(\Aa_k^1)=\Mot(\Spec{k})(-1)[-2]$ and $\Mot_c(\Spec{k})=\Mot(\Spec{k})$, we have the distinguished triangle
\[
\Mot(\Gm_k)(-1)[-2]\to \Mot(\Spec{k})(-1)[-2]\to \Mot(\Spec{k})\to \Mot(\Gm_k)(-1)[-1].
\]
We apply the equivalence of categories $\mathfrak{F}$ and use the fact that
 $\Motr(\parBb(o,1))\simeq\mathfrak{F}(\Mot(\Gm_k))$ and $\Motr(\Spm{K})=\mathfrak{F}\Mot(\Spec{k})$. Hence we get the following triangle in $\RigSH{K}$ :
\[
\Motr(\parBb(o,1))(-1)[-2]\to \Motr(\Spm{K})(-1)[-2]\to \Motr(\Spm{K})\overset{+1}{\to}.
\]
Applying the Tate twist , we get finally the following equality in $\K{\RigSH{K}}$
\[ 
[\Motr(\Spm{K})]=[\Motr(\Spm{K})(1)]+[\Motr(\parBb(o,1))],
\]
which is
\[
\chirpRV([\set{1}]_0]=\chirpRV([\set{1}]_1]-\chirpRV([\RV^{>0}]_1), 
\]
hence
\[
\chirpRV([\set{1}]_1)=\chirpRV([\set{1}]_0+[\RV^{>0}]_1).
\] 
\end{proof}

Recall the isomorphism $\oint : K(\VF_K)\to K(\RV_K[*])/\Isp$. 
\begin{Definition}
Define $\chir$ and $\chirp$ : 
\[\K{\VF}\to \K{\RigSH{K}}\] 
by $\chir=\chirRV\circ\oint$ and $\chirp=\chirpRV\circ\oint$.
\end{Definition}

\begin{Proposition}
\label{prop-chir-QVF}
For any irreducible smooth $k$-variety $X$ of dimension $d$, $f\in \calO_X^\times(X)$ and $r\in \Nn^*$, 
\[
\chir(Q_r^\VF(X,f))=[\Motr(Q_r^\Rig(X,f))(-d)].
\]
\end{Proposition}
\begin{proof}
Since $\oint [Q_r^\VF(X,f)]=[Q_r^\RES(X,f)]_r$, it simply follows from the definition of $\chirRes$. 
\end{proof}

\begin{Proposition}
\label{prop-chir-v-1Delta}
Let $\Delta\subset\Gamma^n$ defined by linearly independent affine equations $l_i>0$ for $i=0,...,n$.
Then 
\[\chir(\val^{-1}(\Delta))=(-1)^n [\Motr(\val^{-1}(\Delta)^\Rig)(-n)].\]
\end{Proposition}
\begin{proof}
It follows from Proposition \ref{prop-chirGam-locclosed}.
\end{proof}

\begin{Theorem}
\label{thm-diagram-chir-int-R-mu}
There are commutative squares
\[
\xymatrixcolsep{5pc}
\xymatrix{
\K{\VF_K} \ar[d]_{\chir} \ar[r]^-{\Theta\circ \calE_c \circ \oint} & \K{\Varmu{k}} \ar[d]^{\chi_{\hat \mu}} \\
\K{\RigSH{K}} \ar[r]_-{\mathfrak{R}} & \K{\QUSH{k}}
}
\]
and
\[
\xymatrixcolsep{5pc}
\xymatrix{
\K{\VF_K} \ar[d]_{\chirp} \ar[r]^-{\Theta\circ \calE \circ \oint} & \K{\Varmu{k}}[\Aa^1_k]^{-1} \ar[d]^{\chi_{\hat \mu}} \\
\K{\RigSH{K}} \ar[r]_-{\mathfrak{R}} & \K{\QUSH{k}}.
}
\]
\end{Theorem}

\begin{proof} We will only show the commutativity of the first diagram, the second being similar. 

We need to show that the following diagram is commutative :
\[
\xymatrixcolsep{5pc}
\xymatrix{
\K{\RV[*]}/\Isp \ar[d]_{\chirRV} \ar[r]^-{\Theta\circ \calE_c} & \K{\Varmu{k}} \ar[d]^{\chi_{\hat \mu}} \\
\K{\RigSH{K}} \ar[r]_-{\mathfrak{R}} & \K{\QUSH{k}}. 
}
\]
It suffices to show that the following diagrams are commutative :

\[
\xymatrixcolsep{5pc}
\xymatrix{
\K{\RES_K[*]} \ar[d]_{\chirRes} \ar[r]^-{\Theta\circ \calE_c} & \K{\Varmu{k}} \ar[d]^{\chi_{\hat \mu}} \\
\K{\RigSH{K}} \ar[r]_-{\mathfrak{R}} & \K{\QUSH{k}}. 
}
\]

and

\[
\xymatrixcolsep{5pc}
\xymatrix{
\K{\Gam} \ar[d]_{\chirGam} \ar[r]^-{\Theta\circ \calE_c} & \K{\Varmu{k}} \ar[d]^{\chi_{\hat \mu}} \\
\K{\RigSH{K}} \ar[r]_-{\mathfrak{R}} & \K{\QUSH{k}}. 
}
\]
For the first one, we already observed that $\chirRes$ does not depend on the grading, hence we need to show that $\mathfrak{R}\circ\chirRes=\chi_{\hat \mu}\circ \Theta$, as morphisms from $!\K{\RES_K}$ to $\K{\QUSH{k}}$.  By Corollary \ref{cor-QRES-generators-KRES}, $!\K{\RES_K}$ is generated by classes of $Q_r^\RV(X,f)$, for $X$ a $k$-variety smooth of pure dimension $d$, $r\in \Nn\backslash \set{0}$ and $f\in \Gam(X,\calO_X^\times)$. The definition of $\chi_{\hat \mu}\circ \Theta$ show that $\chi_{\hat \mu}\circ \Theta(Q_r^\RV(X,f))=[\Mot_{\Gm_k}(Q_r^\gm(X,f))(-\dim(X))]$. From the definition of $\chirRes$, $\chirRes(Q_r^\RV(X,f))=[\Motr(Q_r^\Rig(X,f)(-d))]$ ; from Theorem \ref{thm-ayoub-equi-Qrig-Qan}, $\Motr(Q_r^\Rig(X,f)\simeq \Motr(Q_r^\an(X,f)$, and from the definition of $\mathfrak{R}$, $\mathfrak{R}(\Motr(Q_r^\an(X,f))=\Mot_{\Gm_k}(Q_r^\gm(X,f))$.
 For the second square, for any $X\subset\Gam^n$, $\chirGam(X)=\eu_c(X)[\parBb(o,1)^n(-n)]$ and $\calE_c(X)=\eu_c(X)[\Gm_k^n]$, so it follows from the fact that $\mathfrak{R}[\Motr(\parBb(o,1)]= [\Mot_{\Gm_k}(\Gm_k\times_k \Gm_k)]$.
\end{proof}

\subsection{Motives of tubes}
\label{section-motive-tube}

The aim of this section is to compute $\chir$ for a quasi-compact smooth rigid $K$-variety. We will use semi-stable formal models, and in particular tubes of their branches.

Let $\calX$ a semi-stable formal $R$-scheme and $(D_i)_{i\in J}$ be the branches of its special fiber $\calX_\sigma$.  For any non empty $I\subset J$, let $D_I=\cap_{i\in I} D_i$ 	and $D(I)=\cup_{i\in I} D_i$. Denote also for $I'\subset J\backslash I$, $D_{I}^{\circ I'}=D_I\backslash D(I')$ and if $I'=J\backslash I$, simply $D_{I}^\circ=D_I\backslash D(J\backslash I)$.

Ayoub, Ivorra and Sebag prove the following proposition. 
\begin{Proposition}[{\cite[Theorem 5.1]{AIS}}]
\label{propAIS-MDI=MDI0}
For any non empty $I\subset J$ and $I'\subset I''\subset J\backslash I$, the inclusion $\tube{D_I^{\circ I''}}\hookrightarrow \tube{D_I^{\circ I'}}$ induces an isomorphism
\[\Motr(\tube{D_I^{\circ I''}})\simeq \Motr(\tube{D_I^{\circ I'}}).\]
\end{Proposition}

We will mostly use this proposition in the following particular case. 

\begin{Corollary}
\label{corAIS-MDI-MDI0}
For any non empty $I\subseteq J$, there is an isomorphism
\[ \Motr(\tube{D_I^\circ})\simeq \Motr(\tube{D_I}).\]
\end{Corollary}

\begin{Proposition}
\label{prop-cutting-Mtubes}
Let $\calX$ a semi-stable formal $R$-scheme and $D=\cup_{i\in J'} D_i$ a finite union of its branches. 
Then the following equalities holds in $\K{\RigSH{K}}$
\[
[\Motr(\tube{D})]=\sum_{I\subset J'}(-1)^{\abs{I}-1} [\Motr(\tube{D_I})]
\]
and
\[
[\Motr^\vee(\tube{D})]=\sum_{I\subset J'}(-1)^{\abs{I}-1} [\Motr^\vee(\tube{D_I})].
\]

\end{Proposition}

\begin{proof}
The collection $(D_i)_{i\in J'}$ is a closed cover of $D$, hence by Proposition \ref{prop-tube-admissible-cover}, the $(\tube{D_i})_{i\in J'}$ is an admissible cover of $\tube{D}$. Hence by Mayer-Vietoris distinguished triangle and induction on the cardinal of $I$, we have the result. 
\end{proof}

Using Corollary \ref{corAIS-MDI-MDI0}, we deduce the following formula. 
\begin{Corollary}
\label{cor-cutting-Mtubes}
Under the hypotheses of Proposition \ref{prop-cutting-Mtubes}, we have
\[ [\Motr(\tube{D})]=\sum_{I\subseteq J'}(-1)^{\abs{I}-1} [\Motr(\tube{D_I^\circ})]
\]
and
\[ [\Motr^\vee(\tube{D})]=\sum_{I\subseteq J'}(-1)^{\abs{I}-1} [\Motr^\vee(\tube{D_I^\circ})].
\]
\end{Corollary}

\begin{Theorem}
\label{thm-chir-tube-semistable}
Let $\calX$ a semi-stable formal $R$-scheme of dimension $d$. 
Then 
\[
\chir(\calX_\eta^\VF)=[\Motr(\calX_\eta(-d))].
\]

\end{Theorem}

Still denoting $\calX_\sigma=\cup_{i\in J} D_i$ the irreducible components of $\calX_\sigma$, the special fiber of $X$, we can write $\calX_\eta^\VF$ as a disjoint union
$\calX_\eta^\VF=\overset{.}{\bigcup}_{I\subset J}\tube{D_I^\circ}$,
hence
\[
\chir(\tube{D}_\calX^\VF)=\sum_{I\subset J} \chir(\tube{D_I^\circ}).
\]
In view of the formula of Corollary \ref{cor-cutting-Mtubes}, to prove Theorem \ref{thm-chir-tube-semistable}, it suffices to prove the following proposition.

\begin{Proposition}
\label{prop-chir-single-tube}
Let $\calX$ be a semi-stable $R$-scheme. Then
\[
\chir(\tube{D_I^\circ}^\VF)=(-1)^{\abs{I}-1}[\Motr(\tube{D_I^\circ})(-d)], 
\]
where $d=\dim(\calX_\sigma)$. 
\end{Proposition}
Before proving the proposition, we need a reduction. 
\begin{Lemma}
\label{lem-reduction-chir-single-tube}
To prove Proposition \ref{prop-chir-single-tube}, we can assume 
$\calX=\St^{u^{-1}t}_{D_I^\circ\times_k R, \underline{N}}$,
where $\underline{N}=(N_1,...,N_r)\in (\Nn^\times)^r$ (where $r=\abs{I}$), $u_I\in \calO^\times(D_I^\circ\times_k R)$. 
\end{Lemma}

\begin{proof}
Replace $\calX$ by $\calX\{V,V^{-1}\}$. Using Mayer-Vietoris distinguished triangles, we can also work Zariski locally, hence suppose by Proposition \ref{prop-semistable-etaleproj} that there is an \'etale $R$-morphism
\[ e :\calX \to \calS=\St^{Ut}_{\Spec{R[U,U^{-1}]},\underline{N}}[S_1,...,S_r],\]
where $\underline{N}$ is the type of $\calX$ at $x\in \tube{D_I^\circ}$. The irreducible components of $\calS_\sigma$ are defined by equations $T_i=0$, denote by $C$ their intersection. We have $C=\Spec{k[U,U^{-1},S_1,...,S_r]}$. Up to permuting the $D_i$, we can assume $D_i$ is defined in $\calX_\sigma$ by $T_i\circ e=0$, inducing an \'etale morphism $e_\sigma : D_I\to C$ and a cartesian square of $R$-schemes
\[
\xymatrixcolsep{5pc}
\xymatrix{
D_I \ar[d]_{e_\sigma} \ar[r]^-{} & X \ar[d]^{e} \\
C \ar[r]_-{} & S.
}
\]
 The morphism $e_\sigma$ induces an \'etale morphism of $R$-schemes 
 \[
 D_I\times_k R\to C\times_k R, 
 \]
which itself induces an \' etale $R$-morphism
\[
e' : \calX'=\St^{u_I^{-1}t}_{D_I\times_k R, \underline{N}}\to \calS,
\]
together with a cartesian square
\[
\xymatrixcolsep{5pc}
\xymatrix{
D_I \ar[d]_{e_\sigma} \ar[r]^-{} & X' \ar[d]^{e'} \\
C \ar[r]_-{} & S.
}
\]
The fiber product $\calX\times_\calS \calX'$ hence satisfies
$\calX\times_\calS \calX' \times_\calS  C\simeq D_I\times_C D_I$. Because $e_\sigma : D_I \to C$ is \'etale, the diagonal embedding $D_I\to D_I\times_C D_I$ is an open and closed immersion, hence induces a decomposition $D_I\times_C D_I\simeq D_I\cup F$. Set $\calX''=\calX\times_\calS \calX'\backslash F$. We have two \'etale morphisms
$f : \calX'' \to \calX$ and $f' : \calX'' \to \calX$ such that $f^{-1}(D_I)\simeq D_I$ and $f'^{-1}(D_I)\simeq D_I$, hence we can apply twice Proposition \ref{prop-berth-iso-tubes-etale} to get that
\[
\tube{D_I^\circ}_\calX\simeq \tube{D_I^\circ}_{\calX''}
\]
and 
\[
\tube{D_I^\circ}_{\calX'}\simeq \tube{D_I^\circ}_{\calX''}.
\]

Recall that be replaced $\calX$ by $\calX\{V,V^{-1}\}$ at the begining of the proof. Hence going back to the notation in the statement of the lemma, what we showed is 
\[
\tube{D_I^\circ}_{\calX\set{V,V^{-1}}}\simeq \tube{D_I^\circ}_{\calX'\set{V,V^{-1}}}.
\]
Inspection of this isomorphism show that it induces the required isomorphism
\[
\tube{D_I^\circ}_{\calX}\simeq \tube{D_I^\circ}_{\calX'},
\]
see also \cite[Theorem 2.6.1]{nicaise_tropical_2017}.
\end{proof}
\begin{Remark}
\label{rem-reduction-chir-single-tube}
The proof of Lemma \ref{lem-reduction-chir-single-tube} also gives a definable bijection $\tube{D_I^\circ}_{\calX}^\VF\simeq \tube{D_I^\circ}_{\calX'}^\VF$.
\end{Remark}

\begin{proof}[Proof of Proposition \ref{prop-chir-single-tube}]
We can suppose that we are in the situation of Lemma \ref{lem-reduction-chir-single-tube}, with $\calX=\St^{u^{-1}t}_{D_I^\circ\times_k R, \underline{N}}$. Let $N_I$ be the greatest common divisor of the $N_i$ for $i\in I$. Let $N_i'=N_i/N_I$. As the $N_i'$ are coprime, we can form an $r\times r$ matrix $A\in \mathrm{GL}_n(\Zz)$ which first row is constituted by the $N_i'$. The matrices $A$ and $A^{-1}$ define automorphisms of $\Gm_K^{r,\an}$, hence of $G=D_I^\circ(R) \times \Gm_K^{r,\an}$. As $\tube{D_I^\circ}_\calX$ is a rigid subvariety of $G$, we can consider $W$, its image by $A$. Then $W$ is the locally closed semi-algebraic subset of $G$ defined by

\[
\set{(x,w)\in D_I^\circ(R)\times (K^\times)^r\mid w_1^{N_I}u_I(x)=t, l_1(\val(w))>0,...,l_{r-1}(\val(w))>0},
\]

where the $l_i : \Gam^r \to \Gam$ are  linearly independent affine functions of linear parts with integer coefficients. Hence $W=Q_{N_I}^\VF(D_I^\circ,u_I)\times \val^{-1}(\Delta)$, where $\Delta\subset \Gam^{r-1}$ is defined by equations $l_i>0$ for $i=1,...,r$. 

By Propositions \ref{prop-chir-QVF} and \ref{prop-chir-v-1Delta}, we know that 
\[\chir(Q_{N_I}^\VF(D_I^\circ,u_I))=[\Motr(Q_{N_I}^\Rig(D_I^\circ,u_I))(-d+r-1)]\]
and
\[\chir(\val^{-1}(\Delta))=(-1)^{r-1}[\Motr(\val^{-1}(\Delta)^\Rig)(-r+1)].\]
Hence as $\chir$ is multiplicative, 
\[ \chir(W)=(-1)^{r-1}[\Motr(W^\Rig)(-d)]\]
and applying the isomorphism $A^{-1}$, we get as required \[ \chir(\tube{D_I^\circ}^\VF_X)=(-1)^{\abs{I}-1}[\Motr(\tube{D_I^\circ})(-d)].\]
\end{proof}
The proof of Theorem \ref{thm-chir-tube-semistable} is now complete.

For later use, note that the proofs of Proposition \ref{prop-chir-single-tube} and Lemma \ref{lem-reduction-chir-single-tube} gives the following equality. 
\begin{Corollary}
\label{cor-prop-chir-single-tube} With the notation of Proposition \ref{prop-chir-single-tube}, we have
\[
\Motr(\tube{D_I^\circ})\simeq \Motr(Q_{N_I}^\Rig(D_I^\circ,u_I)\times \parBb(o,1)^{\abs{I}-1}).\]
\end{Corollary}

\subsection{Compatibilities of $\chir$}
\label{section-compatib-chir}

We will now derive consequences of Theorem \ref{thm-chir-tube-semistable}.

\begin{Theorem}
\label{thm-chir-uniquenss-qcsmvarieties}
The morphism $\chir$ is the unique ring morphism
\[
\K{\VF_K}\to \K{\RigSH{K}}
\]
such that for any quasi-compact smooth rigid $K$-variety of pure dimension $d$,
\[
\chir(X^\VF)=[\Motr(X)(-d)].
\]
\end{Theorem}

\begin{proof}
By quantifier elimination in the theory $\ACVF_K$, $\K{\VF_K}$ is generated by classes of smooth affinoid rigid $K$-varieties, which shows uniqueness. 
For the existence, fix $X$ a quasi-compact smooth rigid $K$-variety of pure dimension $d$. We can find $\calX$, a formal $R$-model of $X$ and by Hironaka's resolution of singularities, we can assume $\calX$ is semi-stable. We can now apply Theorem \ref{thm-chir-tube-semistable}.
\end{proof}

\begin{Theorem} There is a commutative diagram
\label{thm-comp-chir-chi_K}
\[
\xymatrixcolsep{5pc}
\xymatrix{
\K{\Var{K}} \ar[d]_{\chi_K} \ar[r]^-{} & \K{\VF_K} \ar[d]^{\chir} \\
\K{\SH{K}} \ar[r]_-{\Rig} & \K{\RigSH{K}}.
}
\]

\end{Theorem}

\begin{proof}
By Nagata's compactification theorem and Hironaka's resolution of singularities, the ring $\K{\Var{K}}$ is generated by classes of smooth projective varieties, hence it suffices to check the compatibility for such a $K$-variety $X$. Denoting $f : X \to K$ the structural morphism, one has by definition 
\[\chi_K([X])=[f_!f^*\un_K]=[f_\sharp f^*\un_K(-d)]=[\Mot_K(X)(-d)],\] where $d=\dim(X)$ and we used $f_!=f_\sharp \circ\mathrm{Th}^{-1}(\Omega_f)$ for $f$ smooth. 
Applying the functor $\Rig$, one needs to show that $\chir([X]^\VF)=[\Motr(X^\an)(-d)]$. 
As $X$ is projective, $X^\an$ is a quasi-compact rigid $K$-variety hence one can find a semi-stable formal model of $X^\an$ over $R$, denote it $\widetilde{X}$.
Hence $X\simeq \widetilde{X}_K\simeq \tube{\widetilde{X}_\sigma}_{\widetilde{X}}$, so by Theorem \ref{thm-chir-tube-semistable}, $\chir([X]^\VF)=[\Motr(X^\Rig)(-d)]$.
\end{proof}

Note that combining Theorems \ref{thm-chir-uniquenss-qcsmvarieties} and \ref{thm-comp-chir-chi_K} gives Theorem \ref{main-thm1}. 

\subsection{A few more realization maps}
\label{section-few-more-realization}

In this section we construct in addition to $\chir$ and $\chirp$ two more realization maps $\chirt$ and $\chirpt$ obtained by considering homological motives with compact support instead of cohomological motives with compact support. 

Recall the ring morphism of Proposition \ref{realisation-motives-is} $\chi_S : \K{\Var{S}}\to \K{\SH{S}}$ sending $[f:X\to S]$ to $[\Mot_{S,c}^\vee(X)]=[f_!f^*\un_S]$. 

Working dually, we can define also a morphism 
$\widetilde{chi_S} : \K{\Var{S}}\to \K{\SH{S}}$ sending $[f:X\to S]$ to $[\Mot_{S,c}(X)]=[f_*f^!\un_S]$. The proof that it respect the scissors relations is similar, using exact triangle {eqn-localitytrianglebisdual} instead of {eqn-localitytrianglebis}. We can also use duality involutions of the following Section \ref{section-duality}. 

Composing with the morphism $\K{\Varmu{k}}\to \K{\Var{\Gm_k}}$, we get a morphism $\widetilde{\chi_{\hat\mu}} : \K{\Varmu{k}} \to \K{\QUSH{\Gm_k}}$, fitting in the following commutative square :
\[
\xymatrixcolsep{5pc}
\xymatrix{
 \K{\Varmu{k}} \ar[d]_{\widetilde{\chi_{\hat\mu}} } \ar[r]^-{} & \K{\Var{k}} \ar[d]^{\widetilde{\chi_k}} \\
\K{\QUSH{k}} \ar[r]_-{1^*} &\K{\SH{k}}.
}
\]

Recall that 
\[
\chir=\mathfrak{F}\circ{\chi_{\hat\mu}}\circ\Theta\circ\calE_c\circ\oint
\;\mathrm{and}\; 
\chirp=\mathfrak{F}\circ{\chi_{\hat\mu}}\circ\Theta\circ\calE\circ\oint.
\]

We can now define 
\[
\chirt=\mathfrak{F}\circ\widetilde{\chi_{\hat\mu}}\circ\Theta\circ\calE_c\circ\oint
\;\mathrm{and}\; 
\chirpt=\mathfrak{F}\circ\widetilde{\chi_{\hat\mu}}\circ\Theta\circ\calE\circ\oint.
\]

Unraveling the definitions, we see that if $X$ is a smooth connected $k$-variety of dimension $d$, $f\in \Gam(X,\calO_X^\times)$, $r\in \Nn^*$ and $\Delta\subset\Gam^n$ an open simplex of dimension $n$, 
\[\chirt(Q_r^\VF(X,f))=[\Motr^\vee(Q_r^\Rig(X,f))(d)],
\]
\[\chirpt(Q_r^\VF(X,f))=[\Motr^\vee(Q_r^\Rig(X,f))],
\]
\[\chirt(\val^{-1}(\Delta))=(-1)^{d}[\Motr^\vee(\val^{-1}(\Delta)^\Rig)(d)],
\]
and
\[\chirpt(Q_r^\VF(X,f))=(-1)^{d}[\Motr^\vee(\val^{-1}(\Delta)^\Rig)].
\]
See the proofs of Propositions \ref{prop-chir-QVF} and \ref{prop-chir-v-1Delta} for details. 

Hence the proof of Theorem \ref{thm-chir-tube-semistable} can be adapted to $\chirp$, $\chirt$ and $\chirpt$, showing in particular that if $X$ is a quasi-compact smooth connected rigid $K$-variety of dimension $d$, 

\[
\chir(X^\VF)=[\Motr(X)(-d)], \quad \chirp(X^\VF)=[\Motr(X)],
\]

\[
\chirt(X^\VF)=[\Motr^\vee(X)(d)], \quad \chirpt(X^\VF)=[\Motr^\vee(X)].
\]

If $X$ is a proper algebraic $K$-variety of structural morphism $f$, since $f_*=f_!$, we have $\Mot_{K,c}^\vee(X)=\Mot_K^\vee(X)$ and $\Mot_{K,c}^\vee(X)=\Mot_K(X)$, hence we can adapt the proof of Theorem \ref{thm-comp-chir-chi_K} to get commutative diagrams similar of Theorem \ref{main-thm1}, the first one being the statement of Theorem \ref{main-thm1}. 

\begin{Proposition}
The squares in the following diagrams commutes :
\[
\xymatrixcolsep{1pc}
\xymatrix{
\K{\Var{K}} \ar[d]_{\chi_K} \ar[r]^-{} &\K{\VF_K} \ar[d]_{\chir} \ar[r]^-{\Theta\circ \calE_c\circ\oint }&\K{\Varmu{k}} \ar[d]^{\chi_{\hat \mu}}\ar[r]^-{} & \K{\Var{k}} \ar[d]^{\chi_k} \\
\K{\mathrm{SH}(K)} \ar[r]_-{\Rig^*} &\K{\mathrm{RigSH}(K)} \ar[r]^-{\simeq}_-{\mathfrak{R}} & \K{\mathrm{QUSH}(k)} \ar[r]_-{1^*} &\K{\mathrm{SH}(k)},
}
\]
\[
\xymatrixcolsep{1pc}
\xymatrix{
\K{\Var{K}} \ar[d]_{\widetilde{\chi_K}} \ar[r]^-{} &\K{\VF_K} \ar[d]_{\chirp}\ar[r]^-{\Theta\circ \calE\circ\oint }&\K{\Varmu{k}}[\eL^{-1}] \ar[d]^{\chi_{\hat \mu}}\ar[r]^-{} & \K{\Var{k}}[\eL^{-1}] \ar[d]^{\chi_k} \\
\K{\mathrm{SH}(K)} \ar[r]_-{\Rig^*} &\K{\mathrm{RigSH}(K)} \ar[r]^-{\simeq}_-{\mathfrak{R}} & \K{\mathrm{QUSH}(k)} \ar[r]_-{1^*} &\K{\mathrm{SH}(k)},
}
\]
\[
\xymatrixcolsep{1pc}
\xymatrix{
\K{\Var{K}} \ar[d]_{\widetilde{\chi_K}} \ar[r]^-{} &\K{\VF_K} \ar[d]_{\chirt} \ar[r]^-{\Theta\circ \calE_c \circ\oint}&\K{\Varmu{k}} \ar[d]^{\widetilde{\chi_{\hat \mu}}}\ar[r]^-{} & \K{\Var{k}} \ar[d]^{\widetilde{\chi_k}} \\
\K{\mathrm{SH}(K)} \ar[r]_-{\Rig^*} &\K{\mathrm{RigSH}(K)} \ar[r]^-{\simeq}_-{\mathfrak{R}} & \K{\mathrm{QUSH}(k)} \ar[r]_-{1^*} &\K{\mathrm{SH}(k)},
}
\]
\[
\xymatrixcolsep{1pc}
\xymatrix{
\K{\Var{K}} \ar[d]_{\chi_K} \ar[r]^-{} &\K{\VF_K} \ar[d]_{\chirpt} \ar[r]^-{\Theta\circ \calE\circ\oint }&\K{\Varmu{k}}[\eL^{-1}] \ar[d]^{\widetilde{\chi_{\hat \mu}}}\ar[r]^-{} & \K{\Var{k}}[\eL^{-1}] \ar[d]^{\widetilde{\chi_k}} \\
\K{\mathrm{SH}(K)} \ar[r]_-{\Rig^*} &\K{\mathrm{RigSH}(K)} \ar[r]^-{\simeq}_-{\mathfrak{R}} & \K{\mathrm{QUSH}(k)} \ar[r]_-{1^*} &\K{\mathrm{SH}(k)}.
}
\]
\end{Proposition}

In particular, we see that $\chir$ and $\chirpt$ agree on the image of $\K{\Var{K}}$ and similarly $\chirp$ and $\chirt$ agree on the image of $\K{\Var{K}}$.

\begin{Remark}[Volume forms]
In addition of the additive morphism $\oint$, Hrushovski and Kazhdan also study the Grothendieck ring of definable sets with volume forms $\K{\mu_\Gamma \VF_K}$. Objects in $\mu_\Gamma \VF_K$ are pairs $(X,\omega)$ with $X\subseteq \VF^\bullet$ a definable set and $\omega : X\to \Gamma$ a definable function. Morphisms are measure preserving definable bijections (up to a set of lower dimension). In this context, they build an isomorphism
\[
\oint^\mu : \K{\mu_\Gamma \VF_K}\to \K{\mu_\Gamma \RV_K}/\mu\Isp,
\]
see \cite[Theorem 8.26]{hrushovski_integration_2006}.

One can further decompose $\K{\mu_\Gamma \RV_K}$ similarly to Proposition \ref{prop-iso-RV-RES-Gam}. Using this Hrushovski and Loeser define in \cite{HL_monodromy} for $m\in \Nn$ morphisms
\[
h_m : \K{\mu_\Gamma \RV^\mathrm{bdd}_K}/\mu\Isp\to \K{\Varmu{k}}_{\mathrm{loc}}
\]
with the $m$ related to considering rational points in $k((t^{1/m}))$. Here, $\mathrm{bdd}$ means we consider only bounded sets. Note that there is an inacurracy in the definition of $h_m$ in \cite{HL_monodromy} since they use \cite[Proposition 10.10 (2)]{hrushovski_integration_2006} which happens to be incorrect. Using the category of bounded sets with volume forms deals with the issue.

We can further compose with the morphism $\mathfrak{F}\circ\chi_{\hat \mu}$ in order to get for each $m\in \Nn^*$ a morphism
\[
\K{\mu_\Gamma \VF^{\mathrm{bdd}}_K}\to \K{\RigSH{K}}.
\]

\end{Remark}

\section{Duality}
\label{section-duality}

The goal of this section is to prove Theorem \ref{main-thm2}. We will adapt Bittner's results on duality in the Grothendieck group of varieties in Section \ref{subsection-duality-involutions} in order to be able to compute in Section \ref{subsection-computation-cohomological-motives} explicitly the cohomological motive of some tubes in terms of homological motives. The last Section \ref{subsection-analytic-milnor-fib} is devoted to an application to the motivic Milnor fiber and the analytic Milnor fiber.

\subsection{Duality involutions}
\label{subsection-duality-involutions}

Bittner developed in \cite{bittner_universal_2004} an abstract theory of duality in the equivariant Grothendieck group of varieties. We recall here some of her results and show that they imply similar results for $\K{\SH{K}}$. 

Using the weak factorization theorem, Bittner prove the following alternative description of $\K{\Var{X}}$. We state if for a variety $S$ above $K$, but it holds for varieties above any field of characteristic zero. 
\begin{Proposition}[{\cite[Theorem 5.1]{bittner_universal_2004}}]
\label{prop-bittner-presentation-K0}
Fix $S$, a $K$-variety. The group $\K{\Var{S}}$ is isomorphic to the abelian group generated by classes of $S$-varieties which are smooth over $k$, proper over $S$, subject to the relations $[\emptyset]_S=0$ and  $[\mathrm{Bl_Y(X)}]_S-[E]_S=[X]_S-[Y]_S$, where $X$ is smooth over $K$, proper over $S$, $Y\subset X$ a closed smooth subvariety, $\mathrm{Bl_Y(X)}$ is the blow-up of $X$ along $Y$ and $E$ is the exceptional divisor of this blow-up. 
\end{Proposition}

For $ f : X\to Y$ a morphism of $S$-varieties, composition with $f$ induce a (group) morphism $f_! : \K{\Var{X}} \to \K{\Var{Y}}$ and pull-back along $f$ induces a (group) morphism $f^* : \K{\Var{Y}}\to \K{\Var{X}}$. Both induces $\calM_S$-linear morphisms $f_! :\calM_X\to \calM_Y$ and $f^*: \calM_Y\to \calM_X$.

\begin{Definition}
We define now a duality operator $\calD_X : \K{\Var{X}}\to \calM_X$ for any $K$-variety $X$. Define $\calD_X([Y])=[Y]\eL^{-\dim(Y)}$ if $Y$ is an $X$-variety proper over $X$, connected and smooth over $K$. In view of Proposition \ref{prop-bittner-presentation-K0}, to show that it induces an unique (group) morphism $\K{\Var{X}}\to \calM_X$, it suffices to show that if $Y\subset Z$ is a closed immersion of $X$-varieties, proper over $X$, smooth and connected over $K$,
\[
[\mathrm{Bl}_Y(Z)]\eL^{-\dim(Z)}-[E]\eL^{-\dim(Z)+1}=[Z]\eL^{-\dim(Z)}-[Y]\eL^{-\dim(Y)},
\]
it holds since $(\eL-1)[E]=(\eL^{\dim(Z)-\dim(Y)}-1)[Y]$. 
\end{Definition}
See \cite[Definition 6.3]{bittner_universal_2004} for details. 

Observe that $\calD_X(\eL)=\eL^{-1}$ and that $\calD_X$ is $\calD_S$-linear, hence $\calD_X$ can be extended as a $\calD_K$-linear morphism $\calD_X : \calM_X\to \calM_X$, which is an involution. 

Although $\calD_X$ is not in general a ring morphism, $\calD_K$ is a ring morphism. 

For $f : X\to Y$, set $f^!=\calD_Xf^*\calD_Y$ and $f_*=\calD_Yf_! \calD_X$. Observe that if $f$ is proper, $f_!=f_*$ and if $f$ is smooth of relative dimension $d$ over $S$, $f^!=\eL^{-d} f^*$. 

Such a duality operator can also be defined in the Grothendieck ring of varieties equipped with a good action of some finite group $G$, see \cite[Sections 7,8]{bittner_universal_2004} for details.

In $\SH{K}$, the internal hom gives also notion of duality, define the duality functor $\mathbb{D}_K(A)=\intHom_{K}(A,\un_K)$. As $\mathbb{D}_K$ is triangulated, it induces an morphism on $\K{\SH{S}}$, still denoted $\mathbb{D}_K$. By \cite[Th\'eor\`eme 2.3.75]{ayoub_six_1}, $\mathbb{D}_K$ is an autoequivalence on constructible objects and its own quasi-inverse. In particular, (compactly supported) (co)homological motives of $S$-varieties are constructible, hence $\mathbb{D}_K$ is an involution on the image of $\K{\Var{K}}$ by $\chi_K$. One can also define more generally for $a : X\to \Spec{K}$, $\mathbb{D}_X(A)=\intHom_X(A,a^!\un_K)$, but we will not use those. By \cite[Th\'eor\`eme 2.3.75]{ayoub_six_1}, $\mathbb{D}_K(\Mot_K(X))=\Mot_K^\vee(X)$ for any $K$-variety $X$. 

The following proposition shows the compatibility between those two duality operators.  
\begin{Proposition}
\label{prop-compatibility-duality}
There is a commutative diagram
\[
\xymatrixcolsep{5pc}
\xymatrix{
\calM_K \ar[d]_{\chi_{K}} \ar[r]^-{\calD_K} & \calM_K \ar[d]^{\chi_{K}} \\
\K{\SH{K}} \ar[r]_-{\mathbb{D}_K} & \K{\SH{K}}.
}
\]
\end{Proposition}
\begin{proof}
It suffices to show that $\chi_K\calD_K([X])=\mathbb{D}_K\chi_K([X])$ for $X$ a connected, smooth and proper $K$-variety, of dimension $d$. Denote $f : X\to \Spec{K}$. As $f$ is smooth and proper, $[\Mot_{K,c}^\vee(X)]=[\Mot_K^\vee(X)]=[\Mot_K(X)(-d)]$. 
We then have 
\[
\chi_K\calD_K([X])=\chi_K([X]\eL^{-d})=[\Mot_{K,c}^\vee(X)(d)]=[\Mot_{K}(X)]
\]
and
\[\mathbb{D}_K\chi_K([X])=\mathbb{D}_K([\Mot_{K,c}^\vee(X)])=\mathbb{D}_K([\Mot_{K}^\vee(X)])=[\Mot_{K}(X)].
\]
\end{proof}

All our duality results will ultimately boils down to the following lemma, which states that normal toric varieties satisfy Poincare duality. It is due to Bittner, see \cite[Lemma 4.1]{bittner_motivic_2005}. The proof rely on toric resolution of singularities and the Dehn-Sommerville equations, see for example \cite{fulton_introduction_1993}. 
\begin{Lemma}
\label{lem-duality-normal-toric-varieties}
Let $X$ an affine toric $K$-variety associated to a simplicial cone, $X\to Y$ be a proper morphism, $G$ a finite group acting on $X$ via the torus with trivial action on $Y$. Then 
\[
\calD_Y([X])=[X]\eL^{-\dim(X)}\in \calM_Y^G.
\]
\end{Lemma}

For the rest of the section, we fix a semi-stable formal $R$-scheme $\calX$ and let $(D_i)_{i\in J}$ be the branches of its special fiber $\calX_\sigma$. Fix $I\subset J$, up to reordering the coordinates, suppose $I=\set{1,...,k}$. Recall that around every closed point $x\in D_I^\circ$, there is a Zariski open neighborhood $\calU$ and regular functions $u_I,x_1,...,x_k$ such that $u_I\in \calO^\times(\calU)$ and $t=u_I x_1^{N_1}...x_k^{N_k}$, with the branch $D_i$ defined by $x_i=0$. Still denote $u_I, x_1,...,x_k$ their reductions to $U=\calU_\sigma$. The various $u_I$ glue to define a section $u_I\in\Gam(D_I^\circ,\calO^\times_{D_I^\circ})$. Recall that $N_I$ is the greatest common divisor of the $N_i$, $i\in I$. 

We already considered (the analytification of) the $K$-variety
\[
Q^\mathrm{geo}_{N_I}(D_I^\circ,u_I)=D_I^\circ\times_kK[V]/(V^{N_I}-tu_I).
\]
In this section, we will denote $\widetilde{D_I^\circ}=Q^\mathrm{geo}_{N_I}(D_I^\circ,u_I)$ to simplify the notations. We will also abuse the notations and still denote $D_I^\circ$, $D_I$, $U$ the base change to $K$ of those varieties. 

Let $\widetilde{D_I}$ be the normalization of $D_I$ in $\widetilde{D_I^\circ}$. 

We also set for $K\subset I$, $\widetilde{D_I}_{\vert D_K}=\widetilde{D_I}\times_{D_I} {D_K}$

\begin{Proposition}
\label{prop-dual-DIo-kvargm}
For every $I\subset K\subset J$, we have $\widetilde{D_I}_{\vert D_K}\simeq \widetilde{D_K}$ and $\calD_{D_I}[\widetilde{D_I}]=\eL^{\abs{I}-d+1}[\widetilde{D_I}]$.
\end{Proposition}

Observe that Bittner's Lemma 5.2 in \cite{bittner_motivic_2005} is analogous, but holds in $\K{\Varmu{k}}$. Since it is not a priori clear that dualities on $\K{\Varmu{k}}$ and $\K{\QUSH{k}}$ are compatible, we cannot apply directly her Lemma. We will nevertheless follow closely her proof.

Combination of Propositions \ref{prop-dual-DIo-kvargm} and \ref{prop-compatibility-duality} yields the following corollary. 
\begin{Corollary}
\label{cor-dual-DI}
For any $I\subset J$ such that $D_I$ is proper, we have the equality in $\K{\Var{K}}$
\[
[\widetilde{D_I}]=\sum_{I\subset K\subset J} [\widetilde{D_K^\circ}],
\]
and $[\mathbb{D}_{K}(\Mot_{K,c}^\vee(\widetilde{D_I}))]=[\Mot_{K,c}^\vee(\widetilde{D_I})(d-\abs{I}+1)]\in \K{\SH{K}}$.
\end{Corollary}
\begin{proof}
The first equality is the first point of Proposition \ref{prop-dual-DIo-kvargm}. For the second equality, since $D_I$ is proper, denoting $f : D_I\times_k\Gm_k\to \Gm$, we have $f_!\calD_{D_I}= \calD_{K} f_!$ hence by the second point of Proposition \ref{prop-dual-DIo-kvargm},
\[
\calD_{K}([\widetilde{D_I}]=[\widetilde{D_I}]\eL^{-d+\abs{I}-1}.
\]
Since $\widetilde{D_I}$ is proper over $K$, we have $\Mot_{K}^\vee(\widetilde{D_I})=\Mot_{K,c}^\vee(\widetilde{D_I})$, hence the result follows from Proposition \ref{prop-compatibility-duality} after applying a Tate twist. 
\end{proof}

\begin{proof}[Proof of Proposition \ref{prop-dual-DIo-kvargm}]
Working inductively on the codimension of $D_K$ in $D_I$ and
up to reordering the branches, we can suppose $I=\set{1,...,k-1}\subset K=\set{1,...,k}$. 
As the statement is local we can work on an open neighborhood $U$ of some closed point of $D_K$ and choose a system of local coordinates $x_1,...,x_d$ on $U$ such that $u_Ix_1^{N_1}...x_{k-1}^{N_{k-1}}=u_Kx_1^{N_1}...x_k^{N_k}=ux_1^{N_1}...x_d^{N_d}$, where $u\in \calO^\times(U)$
and for $i=1,...,k$, $D_i$ is defined by equation $x_i=0$. Define a $\mu_{N_1}$-\'etale cover of $U$ by 
\[
V=U[Z]/(Z^{N_1}-tu).
\]
We consider $V$ as a variety with a $\mu_{N_1}$-action induced by multiplication of $z$ by $\zeta\in \mu_{N_1}$. 

Then $y_1=sx_1,y_2=x_2,...,x_d=y_{d+1}$ is a system of local coordinates on $V$. Shrinking $U$, we can assume the morphism $V \to \Aa^{d+1}_k$ induced by $y_1,...,y_{d+1}$ is \'etale. 
Denote by $F_I$, $F_I^\circ$, $F_K$, $F_K^\circ$ the pull-backs of $D_I$, $D_I^\circ$, $D_K$, $D_K^\circ$.

Denote by $\widetilde{F_I^\circ}$ the following \'etale cover of $F_I^\circ$ :
\[
\widetilde{F_I^\circ}=F_I^\circ[W]/(W^{N_I}-y_{k+1}^{N_{k+1}}...y_d^{N_d}).
\] 
Observe that $\widetilde{F_I^\circ}$ is isomorphic to the fiber product $(F_I^\circ)\times_{D_I^\circ}\widetilde{D_I^\circ}$. The variety $\widetilde{F_I^\circ}$ is equipped with a $\mu_{N_1}$-action, with $\zeta\in\mu_{N_1}$ acting on $w$ by multiplication by $\zeta^{N_1/N_I}$.

Denote by $\widetilde{F_I}$ the normalization of $F_I$ in $\widetilde{F_I^\circ}$ and consider the following diagram:
\[
\xymatrixcolsep{5pc}
\xymatrix{
p^*\widetilde{D_I} \ar[d]_{} \ar[r]^-{} & \widetilde{D_I}  \ar[d]^{} \\
F_I \ar[r]_-{p} & D_I.
}
\]

As $p : F_I\to D_I$ is smooth and $\widetilde{D_I}$ is normal, $p^*\widetilde{D_I}$ is normal. As $p^*\widetilde{D_I}\to F_I$ is finite and surjective, $p^*\widetilde{D_I}$ is isomorphic to $\widetilde{F_I}$. Denoting $\widetilde{F_I}/\mu_{N_1}$ the quotient of $\widetilde{F_I}$ by the $\mu_{N_1}$-action, we then have $\widetilde{D_I}\simeq \widetilde{F_I}/\mu_{N_1}$ and similarly $\widetilde{D_K}\simeq \widetilde{F_K}/\mu_{N_1}$. Hence it suffices to show that $\widetilde{F_I}_{\vert F_K}\simeq \widetilde{F_K}$ and $\calD_{F_I}[\widetilde{F_I}]=\widetilde{F_I}\eL^{-n+k-1}$, both equalities being compatible with the $\mu_{N_1}$-actions. 

Now consider the \'etale morphism $\pi : V\to \Aa^{d+1}_K$. Denoting $z_1,...,z_{d+1}$ the coordinates of $\Aa^{d+1}_K$, define $C_I\subseteq \Aa^{d+1}_{\Gm_k}$ by the equations $z_1=...=z_{k-1}=0$ and $C_I^\circ=C_I\backslash \cup_{j=k,...d+1}\set{z_j=0}$, define similarly $C_K$ and $C_K^\circ$. 

Define an \'etale cover of $C_I^\circ$ by
\[
\widetilde{C_I^\circ}=C_I^\circ[S]/(S^{N_I}-z_k^{N_k}...z_{d+1}^{N_{d+1}}),
\]
define similarly $\widetilde{C_K^\circ}$ and let $\widetilde{C_I}$ and $\widetilde{C_K}$ be the normalizations of $C_I$ and $C_K$ in $\widetilde{C_I^\circ}$ and $\widetilde{C_K^\circ}$. 

We then have a cartesian diagram
\[
\xymatrixcolsep{5pc}
\xymatrix{
\widetilde{F_I}\ar[d]_{} \ar[r]^-{} & \widetilde{C_I}  \ar[d]^{} \\
F_I \ar[r]_-{\pi} & C_I
}
\]
with $\pi$ \'etale, hence it suffices to show that $\widetilde{C_I}_{\vert C_K}\simeq \widetilde{C_K}$ and $\calD_{C_I}[\widetilde{C_I}]=\widetilde{C_I}\eL^{-n+k-1}$, both equalities being compatible with the $\mu_{N_1}$-actions ($\zeta\in \mu_{N_1}$ acts on $\widetilde{C_I^\circ}$ by multiplication of $s$ by $\zeta^{N_1/N_I}$). Indeed, since $\pi$ is \'etale, we have $\pi*\calD_{C_I}=\calD_{F_I\times_k\Gm_k}\pi^*$.

Since the projection $\Aa^{d+1}\to \Aa^{d+1-k+1}$ is smooth, the result now follows from the following Lemmas \ref{lem-tildeCI-restricts-CK} and \ref{lem-tildeCI-dual} which correspond to Lemmas 5.3 and 5.4 in \cite{bittner_motivic_2005}.
\end{proof}

\begin{Lemma}
\label{lem-tildeCI-restricts-CK}
The restriction of the normalization $\widetilde{S}$ of $S=\set{s^N=x_1^{a_1}...x_d^{a_d}}\subset \Aa^1_k\times_k\Aa^d_{\Gm_k}$ to $\set{x_1=0}\subset \Aa^d_{\Gm_k}$ is isomorphic to the normalization $\widetilde{S'}$ of $S'=\set{s^{N'}=x_2^{a_2}...x_d^{N_d}}$, where $N'=\gcd(N,a_1)$. If $N$ divides some $q\in \Nn^*$, then the isomorphism is compatible with the $\mu_q$-actions on $\widetilde{S}$ and $\widetilde{S'}$ where $\zeta\in \mu_q$ acts on $S$ by multiplication of $s$ by $\zeta^{q/N}$ and on $S'$ by multiplication of $s$ by $\zeta^{q/N'}$.
\end{Lemma}
\begin{proof}
Assume first that $N, a_1,...,a_d$ are coprime. Then $S$ is irreducible. 
Let $M$ be the lattice of $\Rr^{d}$ spanned by $\Zz^{d}$ and $v=(a_1/N ,..., a_d/N)$. Set $M^+=M\cap \Rr^d_+$ Then $\widetilde{S}\simeq \Spec{k[M^+]}$. If $M_1:=\set{u\in M\mid u_1=0}$ and $M_1^+=M_1\cap \Rr^d_+$, then $\Spec{k[M_1^+]}$ is the restriction of $\widetilde{S}$ to $\set{x_1=0}$. 

Now consider the lattice $M'$ generated by $v'=(0,a_2/N',...,a_d/N')$ and  $\set{0}\times \Zz^{d-1}$, and set $M'^+=M'\cap \Rr^d_+$. We have $\widetilde{S'}\simeq \Spec{k[M'^+]}$, hence it suffices to show that $M'\simeq M_1$. 

Denote $e_1,...,e_{d}$ the canonical basis of $\Zz^{d}$. Set $k=a_1/N'$ and $l=N/N'$. 
Observe that $v'=N/N'v-a_1/N'e_1=lv-ke_1$, hence $M'\subseteq M_1$. Reciprocally, if $u=\sum_{i=1}^{d}\lambda_i e_i +\mu v\in M_1$, then $\lambda_1+\mu k/l\in \Zz$, hence $\mu'=\mu/l\in \Zz$ (since $k$ and $\ell$ are coprime), hence $u=\sum_{i=2}^{d} \lambda_i e_i + \mu' v'\in M'$. 

The $\mu_q$-actions are compatible, since $s'^{N'}=s^Nx_1^{-a_1}$. 

Back to the general case, let $c$ be the greatest common divisor of $N, a_1,...,a_d$. Define $e=N/c$, $a_i'=a_i/c$, $e'=N'/c$. Let $\widetilde{T}$ be the normalization of $T=\set{s^e=x_1^{a_1'}...x_d^{a_d'}}$ and $\widetilde{T'}$ be the normalization of $T'=\set{s^{e'}=x_2^{a_2'}...x_d^{a_d'}}$. Both $\widetilde{T}$ and $\widetilde{T'}$ are equipped with a $\mu_e$ action as in the statement of the lemma. 

The mapping 
\[
(\zeta,s,x)\in \mu_N\times T\to (\zeta s,x)\in S
\]
induce an isomorphism 
\[
(\mu_N\times \widetilde{T})/\mu_e\simeq \widetilde{S},
\]
where the $\mu_N$ action on $\widetilde{S}$ correspond to the action on $(\mu_N\times \widetilde{T})/\mu_e$ given by multiplication on $\mu_N$. 
Similarly, the mapping 
\[
(\zeta,s,x)\in \mu_N\times T'\to (\zeta^{N/N'} s,x)\in S'
\]
induce an isomorphism 
\[
(\mu_N\times \widetilde{T'})/\mu_e\simeq \widetilde{S'}.
\]

Hence we can apply the first case to $\widetilde{T}$ and $\widetilde{T'}$ to get $\widetilde{S}_{\vert x_1=0}\simeq \widetilde{S'}$ and check that the actions correspond. 
\end{proof}

\begin{Lemma}
\label{lem-tildeCI-dual}
Denote again by $\widetilde{S}$ the normalization of $S=\set{s^N=x_1^{a_1}...x_d^{a_d}}$. Then $\calD_{\Aa^d_K}[\widetilde{S}]=[\widetilde{S}]\eL^{-d}\in\calM_{\Aa^d_K}^{\mu_q}$, $\zeta\in \mu_q$ acting again by multiplication of $s$ by $\zeta^{q/N}$.
\end{Lemma}

\begin{proof}
It is a particular case of Lemma \ref{lem-duality-normal-toric-varieties} when $N, a_1,...,a_d$ are coprime, and the general case follows as in the proof of Lemma \ref{lem-tildeCI-restricts-CK}. 
\end{proof}

\subsection{Computation of cohomological motives}
\label{subsection-computation-cohomological-motives}

Using Corollary \ref{cor-dual-DI}, we can now compute the cohomological motives of $\widetilde{D_I^\circ}$ in terms of their homological motives. 

\begin{Proposition}
\label{prop-DI-cohomo-homo}
For any $I\subset J$, we have
\[
[\Mot_K^\vee(\widetilde{D_I^\circ})]=\sum_{I\subset L\subset J} (-1)^{\abs{L}-\abs{I}} [\Mot_{K}(\widetilde{D_L^\circ}\times_K \Gm_K^{\abs{L}-\abs{I}})(-d+\abs{I}-1)]\in \K{\SH{K}}.
\]
\end{Proposition}
We first prove an auxiliary formula. 
\begin{Lemma}
\label{lem-DI-cohomo-homo}
For any $I\subset J$, we have $[\Mot_K^\vee(\widetilde{D_I^\circ})]=$
\[
[\Mot_{K}(\widetilde{D_I^\circ})(-d+\abs{I}-1)]
+\sum_{I\subsetneq L\subset J} ([\Mot_{K}(\widetilde{D_L^\circ})(-d+\abs{L}-1)]
- [\Mot_K^\vee(\widetilde{D_L^\circ})(\abs{I}-\abs{L})]).
\]
\end{Lemma}

\begin{proof}
By the first point of Corollary \ref{cor-dual-DI} and additivity of $\Mot_{K,c}^\vee(-)$, we have
\begin{equation}\label{eqn-1}
[\Mot_K^\vee(\widetilde{D_I^\circ})]=[\Mot_{K,c}^\vee(\widetilde{D_I})]-\sum_{I\subsetneq L\subset J} [\Mot_{K,c}^\vee(\widetilde{D_L^\circ})].
\end{equation}
As each of the $\widetilde{D_L^\circ}$ is smooth of pure dimension $d-\abs{L}+1$, by Proposition \ref{prop-compactcohom-motive-smooth}, $[\Mot_{K,c}^\vee(\widetilde{D_L^\circ})]=[\Mot_{K}(\widetilde{D_L^\circ})(-d+\abs{L}-1)]$. We apply $\mathbb{D}_{K}$ to equation $1$. By linearity of $\mathbb{D}_{K}$, the fact that $\mathbb{D}_{K}\Mot_{K}(\widetilde{D_L^\circ})(-d+\abs{L}-1)=\Mot_K^\vee(\widetilde{D_L^\circ})(+d-\abs{L}+1)$ and second point of Corollary \ref{cor-dual-DI}, we get
\begin{equation}
[\Mot_K^\vee(\widetilde{D_I^\circ})(d-\abs{I}+1)]=[\Mot_{K,c}^\vee(\widetilde{D_I})(d-\abs{I}+1)]-\sum_{I\subsetneq L\subset J} [\Mot_K^\vee(\widetilde{D_L^\circ})(d-\abs{L}+1)].
\end{equation}
Twisting this equation $d-\abs{I}+1$ times and applying again Corollary \ref{cor-dual-DI} gives the desired equation. 
\end{proof}

\begin{proof}[Proof of Proposition \ref{prop-DI-cohomo-homo}]
We work by induction on $d-\abs{I}+1$. If $d-\abs{I}+1<0$, then $\widetilde{D_I^\circ}$ is empty and there is nothing to show. If $d-\abs{I}+1=0$, the the formula boils down to
\[
[\Mot_K^\vee(\widetilde{D_I^\circ})]=[\Mot_{K}(\widetilde{D_I^\circ})],
\]
which holds since $\widetilde{D_I^\circ}$ is of dimension 0. 

Suppose now the proposition holds for any $L$ with $\abs{L}>\abs{I}$. Let $r=d-\abs{I}+1$. 

By Lemma \ref{lem-DI-cohomo-homo}, 
\[
[\Mot_K^\vee(\widetilde{D_I^\circ})]=[\Mot_{K}(\widetilde{D_I^\circ})(-r)]
+\sum_{I\subsetneq L\subset J} ([\Mot_{K}(\widetilde{D_L^\circ})(-d+\abs{L}-1)]
- [\Mot_K^\vee(\widetilde{D_L^\circ})(\abs{I}-\abs{L})]).
\]
Applying the induction hypothesis to the $[\Mot_K^\vee(\widetilde{D_L^\circ})]$, we get

\begin{multline*}
[\Mot_K^\vee(\widetilde{D_I^\circ})]=[\Mot_{K}(\widetilde{D_I^\circ})(-r)]
+\sum_{I\subsetneq L\subset J} \bigg ([\Mot_{K}(\widetilde{D_L^\circ})(-d+\abs{L}-1)]\\
 - \sum_{L\subset L'\subset J}(-1)^{\abs{L'}-\abs{L}}[\Mot_{K}(\widetilde{D_{L'}^\circ}\times_K\Gm_K^{\abs{L'}-\abs{L}})(-d+\abs{I}-1)]{\bigg )} .
\end{multline*}
Interverting the sums, we get
\begin{multline*}
[\Mot_K^\vee(\widetilde{D_I^\circ})]=[\Mot_{K}(\widetilde{D_I^\circ})(-r)]
+\sum_{I\subsetneq L\subset J} [\Mot_{K}(\widetilde{D_L^\circ})(-d+\abs{L}-1)]\\
 - \sum_{I\subsetneq L'\subset J}\sum_{i=0}^{\abs{L'}-\abs{I}-1}\binom{\abs{L'}-\abs{I}}{\abs{L'}-\abs{I}-i}(-1)^i[\Mot_{K}(\widetilde{D_{L'}^\circ}\times_K\Gm_K^{i})(-r)].
\end{multline*}
Regrouping the terms, we get 
\begin{multline}
\label{eqn-pf-MveeDI}
[\Mot_K^\vee(\widetilde{D_I^\circ})]=[\Mot_{K}(\widetilde{D_I^\circ})(-r)]
+\sum_{I\subsetneq L'\subset J} [\Mot_{K}(\widetilde{D_{L'}^\circ})(-r)]\\
\cdot\bigg ( [\Mot_{K}(\un)(\abs{L'}-\abs{I})]- \sum_{i=0}^{\abs{L'}-\abs{I}-1}\binom{\abs{L'}-\abs{I}}{i}(-1)^i[\Mot_{K}(\Gm_K^i)]\bigg ).
\end{multline}
We need to compute the expression inside the big brackets. We have

\begin{multline*}
[\Mot_{K}(\un)(\abs{L'}-\abs{I})]- \sum_{i=0}^{\abs{L'}-\abs{I}-1}\binom{\abs{L'}-\abs{I}}{i}(-1)^i[\Mot_{K}(\Gm_K^i)]\\
=[\Mot_{K}(\un)(\abs{L'}-\abs{I})]+(-1)^{\abs{L'}-\abs{I}}[\Mot_{K}(\Gm_K^{\abs{L'}-\abs{I}})]-\big ( [\Mot_{K}(\un)]-[\Mot_{K}(\Gm_K)]\big )^{\abs{L'}-\abs{I}}\\
=(-1)^{\abs{L'}-\abs{I}}[\Mot_{K}(\Gm_K^{\abs{L'}-\abs{I}})]
\end{multline*}
because $[\Mot_{K}(\Gm_K)]=[\Mot_{K}(\un)]-[\Mot_{K}(\un)(1)]$. Injecting in \ref{eqn-pf-MveeDI} gives the required equation for $[\Mot_K^\vee(\widetilde{D_I^\circ})]$. 
\end{proof}

\begin{Proposition}
\label{prop-chir-proper-cohomo}
Let $\calX$ be a semi-stable formal $R$-scheme and $(D_i)_{i\in I}$ the reduced irreducible components of its special fiber. Let $J'\subset J$ such that for every $i\in J'$, $D_i$ is proper. Then setting $D=\bigcup_{i\in J'} D_i$, we have
\[
\chir(\tube{D}_\calX^\VF)=[\Motr^\vee(\tube{D}_\calX)].
\]

\end{Proposition}

\begin{proof}

By additivity of $\chir$, Proposition \ref{prop-chir-single-tube} and Corollary \ref{cor-prop-chir-single-tube}, we have
\begin{align}
\chir(\tube{D}_\calX^\VF)&=\sum_{\underset{I\cap J'\neq \emptyset}{I\subset J}} \chir(\tube{D_I^\circ}^\VF)\notag\\
&=\sum_{\underset{I\cap J'\neq \emptyset}{I\subset J}} (-1)^{\abs{I}-1}[\Motr(\tube{D_I^\circ})(-d)]\notag\\
&=\sum_{\underset{I\cap J'\neq \emptyset}{I\subset J}} (-1)^{\abs{I}-1}[\Motr(Q_{N_I}^\Rig(D_I^\circ,u_I)\times \parBb(o,1)^{\abs{I}-1})(-d)].\label{eqn-chirDfinal}
\end{align}

We will relate the cohomological motive of the tube to this formula using the duality relations proven above.  
By Corollary \ref{cor-cutting-Mtubes}, we have
\begin{equation}
\label{eqn-MveeD}
[\Motr^\vee(\tube{D})]=\sum_{I\subseteq J'}(-1)^{\abs{I}-1} [\Motr^\vee(\tube{D_I^\circ})].
\end{equation}

By Corollary \ref{cor-prop-chir-single-tube}, 
\begin{align}
\label{eqn-MveeDI}
[\Motr^\vee(\tube{D_I^\circ})]&= [\Motr^\vee(Q_{N_I}^\Rig(D_I^\circ,u_I)\times \parBb(o,1)^{\abs{I}-1})]\notag\\
&=[\Motr^\vee(Q_{N_I}^\Rig(D_I^\circ,u_I))]\cdot[\Motr^\vee(\parBb(o,1)^{\abs{I}-1})].
\end{align}

Combining Equations \ref{eqn-MveeD} and \ref{eqn-MveeDI}, we get
\begin{equation}
\label{eqn-RMveeD}
[\Motr^\vee(\tube{D})]=\sum_{I\subseteq J'}(-1)^{\abs{I}-1} [\Motr^\vee(Q_{N_I}^\Rig(D_I^\circ,u_I))][\Motr^\vee(\parBb(o,1^{\abs{I}-1})].
\end{equation}

The analytifcation of $\widetilde{D_I^\circ}$ is $Q_{N_I}^\an(D_I^\circ,u_I)$, hence by Theorem \ref{thm-ayoub-equi-Qrig-Qan}, $\Rig^*\Mot_K(\widetilde{D_I^\circ})=\Motr^\vee(Q_{N_I}^\Rig(D_I^\circ,u_I))$ and similarly for cohomological motives. 

As $D_I$ is proper, we can apply Proposition \ref{prop-DI-cohomo-homo} which gives after applying $\Rig^*$

 \begin{multline}
[\Motr^\vee(Q_{N_I}^\Rig(D_I^\circ,u_I))]=\\\sum_{I\subset L\subset J} (-1)^{\abs{L}-\abs{I}} [\Motr(Q_{N_L}^\Rig(D_L^\circ,u_L)\times_K \parBb(o,1)^{\abs{L}-\abs{I}})(-d+\abs{I}-1)].
\end{multline}

We also know that $[\Motr^\vee(\parBb(o,1))]=-[\Motr(\parBb(o,1))(-1)]$.  With these two remarks, Equation \ref{eqn-RMveeD} yields
\begin{align}
[\Motr^\vee(\tube{D})]&=\sum_{\emptyset\neq I\subset J'} \sum_{I\subset L\subset J} (-1)^{\abs{L}-\abs{I}} [\Motr(Q_{N_L}^\Rig(D_L^\circ,u_L)\times_K \parBb(o,1)^{\abs{L}-1})(-d)]\notag\\
&=\sum_{\underset{L\cap J'\neq \emptyset}{L\subset J}}(-1)^{\abs{L}-1}[\Motr(Q_{N_L}^\Rig(D_L^\circ,u_L)\times_K \parBb(o,1)^{\abs{L}-1})(-d)]\notag\\
&\hspace{2cm}\cdot\sum_{i=1}^{\abs{L\cap J'}} \binom{\abs{L\cap J'}}{i}(-1)^{i-1}\notag\\
&=\sum_{\underset{L\cap J'\neq \emptyset}{L\subset J}}(-1)^{\abs{L}-1}[\Motr(Q_{N_L}^\Rig(D_L^\circ,u_L)\times_K \parBb(o,1)^{\abs{L}-1})(-d)]\label{eqn-RMveefinal}.
\end{align}
Comparing to \ref{eqn-chirDfinal} gives the desired
\[
\chir(\tube{D}^\VF)=[\Motr^\vee(\tube{D})].
\]
\end{proof}

Proposition \ref{prop-chir-proper-cohomo} imply the following theorem, which is Theorem \ref{main-thm2} of the introduction. All we need to do is choosing an semi-stable formal $R$-scheme $\calY$ over $\calX$ such that $\calY\to\calX$ is a composition of admissible blow-ups. Hence the induced morphism at the level of special fibers is  proper and we can apply Proposition  \ref{prop-chir-proper-cohomo}.

\begin{Theorem}
\label{thm-chir-proper-cohomo}
Let $X$ be a quasi-compact smooth rigid variety, $\calX$ an formal $R$-model of $X$, $D$ a locally closed and proper subset of its special fiber $\calX_\sigma$. 
Then
\[
\chir(]D[^\VF)=[\Motr^\vee(]D[)].
\]
In particular, if $X$ is a smooth and proper rigid variety,
\[
\chir([X^\VF])=[\Motr^\vee(X)].
\]
\end{Theorem}

\subsection{Analytic Milnor fiber}
\label{subsection-analytic-milnor-fib}
It is suggested by Ayoub, Ivorra and Sebag in \cite[Remark 8.15]{AIS} that one should be able to recover their comparison result betwen the motivic Milnor fiber and the cohomological motive of the analytic Milnor fiber using a morphism similar to $\chir$. We show below that it is indeed the case and moreover generalize their comparison result to an equivariant setting.

Let $X$ a smooth  $k$-variety and $f : X \to \Aa^1_k$ a non constant regular function. Base change to $R$ makes of $X$ an $R$-scheme. Denote $\calX_f$ the formal completion of $X$ with respect to $(t)$. Its special fiber $\calX_{f,\sigma}$ is the zero locus of $f$ in $X$. For any closed point $x\in \calX_{f,\sigma}$, denote by $\calF_{f,x}$ the tube of $\set{x}$ in $\calX_f$. It is the analytic Milnor fiber. It is a rigid subvariety of $\calX_{f,\eta}$, the analytic nearby cycles. 

Consider an embedded resolution of singularities of $\calX_{f,\sigma}$ in $X$. It is a proper birationnal morphism $h : Y \to X$ such that $h^{-1}(\calX_{f,\sigma})$ is a smooth normal crossing divisor. Denote by $(E_i)_{i\in J}$ the reductions of its (smooth) irreducible components and $N_i\in \Nn^*$ the multiplicity of $E_i$ in $h^{-1}(\calX_{f,\sigma})$. 

For any non empty $I\subset J$, denote by $E_I=\cap_{i\in I} E_i$ and $E_I^\circ=E_I\backslash\cup_{j\in J\backslash I} E_j$. 

Define as follows the \'etale cover $\widetilde{E_I^\circ}$ of $E_I^\circ$. Let $N_I$ be the greatest common divisor of the $N_i$, for $i\in I$. Working locally on some open neighborhood $U$ of $E_I^\circ$ in $Y$, we can assume that $E_i$ is defined by equation $t_i=0$, for some $t_i\in \calO(U)$ and that on $U$, $f=u_It_{i_1}^{N_{i_1}}... t_{i_r}^{N_{i_r}}$, with $u_I\in \calO(U)^\times$. Then set $\widetilde{E_I^\circ\cap U}=\set{(v,x)\mid x\in U, v^{N_I}=u}$.

Recall the motivic Milnor fiber, defined by Denef and Loeser, see for example \cite{denef_loeser_lefschetz} or \cite{denef_loser_motivic_igusa, denef_germs_1999, denef_geometry_2001}, \cite{loeser_seattle_2009}. In an equivariant setting, it is defined for any closed point $x\in \calX_{f,\sigma}$ by the formula
\[
\psi_{f,x}=\sum_{\emptyset \neq I \subseteq J} (-1)^{\abs{I}-1}[\Gm_k^{\abs{I}-1}] [\widetilde{E_I^\circ}\cap h^{-1}(x)]\in \K{\Varmu{k}}[\eL^{-1}].
\]
In particular, they show that this formula is independent of the chosen resolution $h$.

\begin{Proposition}
\label{prop-computation VF-Varmu-Milnor-fiber}
For any closed point $x\in \calX_{f,\sigma}$, 
\[
\Theta\circ \calE_c\circ\oint (\calF_{f,x}^{\VF})=\psi_{f,x}\in \K{\Varmu{k}}.
\]
\end{Proposition}

\begin{proof}
The embedded resolution of $\calX_{f,\sigma}$ induces an admissible morphism 
$h : Y \to X$, hence $\tube{h^{-1}(x)}_Y\simeq \tube{\set{x}}_X$. Up to changing $h$, we can suppose $h^{-1}(x)$ is a divisor $E=\cup_{i\in J'}{E_i}$ in $\calY_\sigma=\cup_{i\in J}D_i$, with $I'\subset I$.
Then we have 
\[\calF_{f,x}^{\VF}=\overset{.}{\bigcup_{\underset{I\cap J'\neq \emptyset}{I\subset J}}} \tube{E_I^\circ}^\VF.\]
We want to show that 
\[
\Theta\circ\calE_c\circ\oint (\tube{E_I^\circ}_X)=(-1)^{\abs{I}-1}[\Gm_k^{\abs{I}-1}\times_k\widetilde{E_I^\circ}].
\]
By Remark \ref{rem-reduction-chir-single-tube} following Lemma \ref{lem-reduction-chir-single-tube}, we can suppose $X=\St^{u_I^{-1}t}_{E_I^\circ\times_k R, \underline{N}}$,
where $\underline{N}=(N_1,...,N_r)\in (\Nn^\times)^r$ (where $r=\abs{I}$), $u_I\in \calO(E_I^\circ\times_k R)^\times$. But now, as in the proof of Proposition \ref{prop-chir-single-tube}, $\tube{E_I^\circ}_X$ is definably isomorphic to  $Q_{N_I}^\VF(E_I^\circ,u)\times \val^{-1}(\Delta)$. Hence $\oint \tube{E_I^\circ}_X= [Q_{N_I}^\RV(E_I^\circ,u_I)\times \valrv^{-1}(\Delta)]_d$. We have $\Theta Q_{N_I}^\RV(E_I^\circ,u_I)=[\widetilde{E_I^\circ}]$ and 
\[
\Theta \circ \calE_c\valrv^{-1}(\Delta)=\eu_c(\Delta)[\Gm_k]^{\abs{I}-1}=(-1)^{\abs{I}-1}[\Gm_k]^{\abs{I}-1}.
\]
Putting pieces together and by linearity of $\Theta \circ \oint$, we get
\[
\cdot\Theta \circ \calE_c\circ \oint \calF_{f,x}^{\VF}={\sum_{\underset{I\cap J'\neq \emptyset}{I\subset J}}} (-1)^{\abs{I}-1}[\Gm_k^{\abs{I}-1}\times_k\widetilde{E_I^\circ}]=\psi_{f,x}\in \K{\Varmu{k}}.
\]
\end{proof}

\begin{Remark}
In the literature, the motivic Milnor is defined in the localization of $\K{\Varmu{k}}$ by $\eL=[\Aa^1_k]$. Such a localization in non injective if $k=\mathbb{C}$, see Borisov \cite{Borisov_eL_zero_divisor}.  However,
 Proposition \ref{prop-computation VF-Varmu-Milnor-fiber} show that it is well defined in  $\K{\Varmu{k}}$. The same fact is proven by Nicaise and Payne in \cite[Corollary 2.6.2]{nicaise_tropical_2017} using a computation similar to ours. 
\end{Remark}

From Theorem \ref{thm-chir-proper-cohomo}, we deduce the following corollary.
\begin{Corollary}
\label{cor-chir-MilnorFib-cohomo}
For any closed point $x\in \calX_{f,\sigma}$, 
\[
\chir(\calF_{f,x}^\VF)=[\Motr^\vee(\calF_{f,x})].
\]
\end{Corollary}

Combining Corollary \ref{cor-chir-MilnorFib-cohomo} and Proposition \ref{prop-computation VF-Varmu-Milnor-fiber} with Theorem \ref{thm-diagram-chir-int-R-mu} gives the following result. 
\begin{Corollary}
\label{cor-comparison-MveeFx-psifx}
For any closed point $x\in \calX_{f,\sigma}$, we have
\[
[\mathfrak{R}\Motr^\vee(\calF_{f,x})]=\chi_{\hat\mu}(\psi_{f,x})\in \K{\QUSH{k}}.
\]
\end{Corollary}
Recall that $\psi_{f,x}$ is the Denef-Loeser motivic Milnor fiber. 
It is a generalization of Corollary 8.8 of Ayoub, Ivorra and Sebag \cite{AIS} at an equivariant level. They show the same equality, but in $\K{\SH{k}}$, hence one deduce their result from Corollary \ref{cor-chir-MilnorFib-cohomo} using Lemma \ref{lem-com-diag-varmu-var-QUSH-SH}.

\end{document}